\documentclass{amsart}
\usepackage{ytableau}

\newtheorem{thm}{Theorem}[section]
\newtheorem{lem}[thm]{Lemma}
\newtheorem{prop}[thm]{Proposition}
\newtheorem{cor}[thm]{Corollary}
\theoremstyle{definition}
\newtheorem{defn}[thm]{Definition}
\theoremstyle{remark}
\newtheorem{rem}[thm]{Remark}

\DeclareMathOperator{\dep}{dep}

\begin{document}

\title{Generating functions for sums of polynomial multiple zeta values}

\author{Minoru Hirose}
\address[Minoru Hirose]{Faculty of Mathematics, Kyushu University
 744, Motooka, Nishi-ku, Fukuoka, 819-0395, Japan}
\email{m-hirose@math.kyushu-u.ac.jp}

\author{Hideki Murahara}
\address[Hideki Murahara]{Nakamura Gakuen University Graduate School,
 5-7-1, Befu, Jonan-ku, Fukuoka, 814-0198, Japan} 
\email{hmurahara@nakamura-u.ac.jp}

\author{Shingo Saito}
\address[Shingo Saito]{Faculty of Arts and Science, Kyushu University,
 744, Motooka, Nishi-ku, Fukuoka, 819-0395, Japan}
\email{ssaito@artsci.kyushu-u.ac.jp}


\begin{abstract}
 The sum formulas for multiple zeta(-star) values and symmetric multiple zeta(-star) values bear a striking resemblance.
 We explain the resemblance in a rather straightforward manner using an identity that involves the Schur multiple zeta values.
 We also obtain the sum formula for polynomial multiple zeta(-star) values in terms of generating functions,
 simultaneously generalizing the sum formulas for multiple zeta(-star) values and symmetric multiple zeta(-star) values.
\end{abstract}

\maketitle

\tableofcontents

\section{Introduction}
\subsection{Multiple zeta(-star) values and their sum formula}
An \emph{index} is a finite sequence of positive integers, including the empty sequence $\emptyset$.
If $\boldsymbol{k}=(k_1,\dots,k_r)$ is an index,
then we define its \emph{weight} by $\lvert\boldsymbol{k}\rvert=k_1+\dots+k_r$ and its \emph{depth} by $\dep\boldsymbol{k}=r$.
We say that an index is \emph{admissible} if either it is empty or its last component is greater than $1$.

If $\boldsymbol{k}=(k_1,\dots,k_r)$ is an admissible index, then we define the \emph{multiple zeta value} and \emph{multiple zeta-star value} by
\[
 \zeta(\boldsymbol{k})=\sum_{1\le m_1<\dots<m_r}\frac{1}{m_1^{k_1}\dotsm m_r^{k_r}},\qquad
 \zeta^{\star}(\boldsymbol{k})=\sum_{1\le m_1\le\dots\le m_r}\frac{1}{m_1^{k_1}\dotsm m_r^{k_r}}
\]
respectively, where we understand that $\zeta(\emptyset)=\zeta^{\star}(\emptyset)=1$.
The multiple zeta(-star) values are known to satisfy a large number of relations, of which one of the most well-known is the \emph{sum formula}.
The sum formula asserts that the multiple zeta(-star) values of fixed weight and depth add up to an integer multiple of the Riemann zeta value:
\begin{thm}[sum formula for multiple zeta(-star) values; Granville~\cite{Gra97}, Zagier]\label{thm:sum_formula_MZV}
 If $r$ is a nonnegative integer and $w$ is an integer with $w\ge r+2$, then
 \[
  \sum_{\substack{k_1+\dots+k_r+a=w\\k_1,\dots,k_r\ge1\\a\ge2}}\zeta(k_1,\dots,k_r,a)=\zeta(w),\qquad
  \sum_{\substack{k_1+\dots+k_r+a=w\\k_1,\dots,k_r\ge1\\a\ge2}}\zeta^{\star}(k_1,\dots,k_r,a)=\binom{w-1}{r}\zeta(w).
 \]
\end{thm}

\subsection{Regularization for multiple zeta(-star) values}
Let $\mathcal{Z}$ denote the $\mathbb{Q}$-linear space spanned by the multiple zeta values.
As illustrated by
\begin{align*}
 \zeta^{\star}(k_1,k_2)
 &=\sum_{1\le m_1\le m_2}\frac{1}{m_1^{k_1}m_2^{k_2}}
  =\Biggl(\sum_{1\le m_1<m_2}+\sum_{1\le m_1=m_2}\Biggr)\frac{1}{m_1^{k_1}m_2^{k_2}}\\
 &=\zeta(k_1,k_2)+\zeta(k_1+k_2),
\end{align*}
the multiple zeta-star values are sums of multiple zeta values and therefore belong to $\mathcal{Z}$.
Moreover, as illustrated by
\begin{align*}
 \zeta(k)\zeta(l)
 &=\Biggl(\sum_{m=1}^{\infty}\frac{1}{m^k}\Biggr)\Biggl(\sum_{m=1}^{\infty}\frac{1}{m^l}\Biggr)
  =\Biggl(\sum_{1\le m_1<m_2}+\sum_{1\le m_2<m_1}+\sum_{1\le m_1=m_2}\Biggr)\frac{1}{m_1^km_2^l}\\
 &=\zeta(k,l)+\zeta(l,k)+\zeta(k+l),
\end{align*}
the space $\mathcal{Z}$ is closed under multiplication, thereby being a $\mathbb{Q}$-algebra.
Ihara, Kaneko, and Zagier~\cite{IKZ06} employed a method, called \emph{regularization}, for defining the multiple zeta(-star) values for non-admissible indices as elements of the polynomial algebra $\mathcal{Z}[T]$,
by assuming that those relations illustrated above hold even for non-admissible indices and setting $\zeta(1)=T$.
For example, by
\[
 \zeta^{\star}(2,1)=\zeta(2,1)+\zeta(3),\qquad
 \zeta(2)\zeta(1)=\zeta(2,1)+\zeta(1,2)+\zeta(3),
\]
we infer that $\zeta(2,1)=\zeta(2)T-\zeta(1,2)-\zeta(3)$ and $\zeta^{\star}(2,1)=\zeta(2)T-\zeta(1,2)$ in $\mathcal{Z}[T]$.
\begin{rem}
 The reader familiar with regularization is reminded that this paper deals only with harmonic regularization, as opposed to shuffle regularization.
\end{rem}

\subsection{Symmetric multiple zeta(-star) values and their sum formula}
If $\boldsymbol{k}=(k_1,\dots,k_r)$ is an index, then we define
\begin{align*}
 \zeta_{S}(\boldsymbol{k})&=\sum_{i=0}^{r}(-1)^{k_{i+1}+\dots+k_r}\zeta(k_1,\dots,k_i)\zeta(k_r,\dots,k_{i+1}),\\
 \zeta_S^{\star}(\boldsymbol{k})&=\sum_{i=0}^{r}(-1)^{k_{i+1}+\dots+k_r}\zeta^{\star}(k_1,\dots,k_i)\zeta^{\star}(k_r,\dots,k_{i+1}).
\end{align*}
Although $\zeta_{S}(\boldsymbol{k})$ and $\zeta_S^{\star}(\boldsymbol{k})$ a priori belong to $\mathcal{Z}[T]$,
it turns out that they have constant terms only and so belong to $\mathcal{Z}$.
Kaneko and Zagier~\cite{KZ19} defined the \emph{symmetric multiple zeta(-star) values} as $\zeta_S(\boldsymbol{k}),\zeta_S^{\star}(\boldsymbol{k})\bmod\zeta(2)$ in $\mathcal{Z}/\zeta(2)\mathcal{Z}$, and the second author~\cite{Mur15} established the sum formula for symmetric multiple zeta(-star) values:
\begin{thm}[sum formula for symmetric multiple zeta(-star) values; Murahara~\cite{Mur15}]\label{thm:sum_formula_SMZV}
 If $r$ and $s$ are nonnegative integers and $w$ is an integer with $w\ge r+s+2$, then
 \begin{align*}
  \sum_{\substack{k_1+\dots+k_r+a+l_1+\dots+l_s=w\\k_1,\dots,k_r,l_1,\dots,l_s\ge1\\a\ge2}}\zeta_S(k_1,\dots,k_r,a,l_1,\dots,l_s)&\equiv\biggl(-(-1)^r\binom{w-1}{r}+(-1)^s\binom{w-1}{s}\biggr)\zeta(w),\\
  \sum_{\substack{k_1+\dots+k_r+a+l_1+\dots+l_s=w\\k_1,\dots,k_r,l_1,\dots,l_s\ge1\\a\ge2}}\zeta_S^{\star}(k_1,\dots,k_r,a,l_1,\dots,l_s)&\equiv\biggl((-1)^s\binom{w-1}{r}-(-1)^r\binom{w-1}{s}\biggr)\zeta(w)
 \end{align*}
 modulo $\zeta(2)\mathcal{Z}$.
\end{thm}

\begin{rem}
 Note that our convention on the order of arguments is opposite to that of \cite{Mur15}.
\end{rem}

Although Theorems~\ref{thm:sum_formula_MZV} and \ref{thm:sum_formula_SMZV} bear a striking resemblance,
no good reason has been offered thus far.
We shall give an identity (Proposition~\ref{prop:relation_between_sum_formulas}) that together with a generalization of Theorem~\ref{thm:sum_formula_MZV} implies Theorem~\ref{thm:sum_formula_SMZV},
and so we have probably succeeded in explaining the resemblance to some extent.

\subsection{Restatement of the sum formulas in terms of generating functions}
We restate Theorems~\ref{thm:sum_formula_MZV} and \ref{thm:sum_formula_SMZV} in terms of generating functions.
Define
\[
 \psi_1(W)=\sum_{k=2}^{\infty}\zeta(k)W^{k-1}\in\mathcal{Z}[[W]].
\]

\begin{rem}
 Our $\psi_1(W)$ is reminiscent of the digamma function, which satisfies
 \[
  \psi(z+1)=-\gamma-\sum_{k=2}^{\infty}\zeta(k)(-z)^{k-1}.
 \]
\end{rem}

Then Theorems~\ref{thm:sum_formula_MZV} and \ref{thm:sum_formula_SMZV} can be rephrased as follows (proofs will be given in Subsection~\ref{subsec:proof_props}):
\begin{prop}\label{prop:gen_func_zeta(k,a)}
 We have
 \begin{align*}
  \sum_{\substack{\boldsymbol{k}\\a\ge2}}\zeta(\boldsymbol{k},a)A^{\dep\boldsymbol{k}}W^{\lvert\boldsymbol{k}\rvert+a}&=\frac{W}{1-A}(\psi_1(W)-\psi_1(AW)),\\
  \sum_{\substack{\boldsymbol{k}\\a\ge2}}\zeta^{\star}(\boldsymbol{k},a)A^{\dep\boldsymbol{k}}W^{\lvert\boldsymbol{k}\rvert+a}&=W(\psi_1((1+A)W)-\psi_1(AW))
 \end{align*}
 in $\mathcal{Z}[A][[W]]$.
\end{prop}

\begin{prop}\label{prop:gen_func_zeta_fin(k,a,l)}
 We have
 \begin{align*}
  &\sum_{\substack{\boldsymbol{k},\boldsymbol{l}\\a\ge2}}\zeta_S(\boldsymbol{k},a,\boldsymbol{l})A^{\dep\boldsymbol{k}}B^{\dep\boldsymbol{l}}W^{\lvert\boldsymbol{k}\rvert+a+\lvert\boldsymbol{l}\rvert}\\
  &\equiv-\frac{W}{1-B}(\psi_1((1-A)W)-\psi_1((B-A)W))+\frac{W}{1-A}(\psi_1((1-B)W)-\psi_1((A-B)W)),\\
  &\sum_{\substack{\boldsymbol{k},\boldsymbol{l}\\a\ge2}}\zeta_S^{\star}(\boldsymbol{k},a,\boldsymbol{l})A^{\dep\boldsymbol{k}}B^{\dep\boldsymbol{l}}W^{\lvert\boldsymbol{k}\rvert+a+\lvert\boldsymbol{l}\rvert}\\
  &\equiv\frac{W}{1+B}(\psi_1((1+A)W)-\psi_1((A-B)W))-\frac{W}{1+A}(\psi_1((1+B)W)-\psi_1((B-A)W))
 \end{align*}
 modulo $\zeta(2)\mathcal{Z}$ in $\mathcal{Z}[A,B][[W]]$.
\end{prop}

\subsection{Polynomial multiple zeta(-star) values}
If $\boldsymbol{k}=(k_1,\dots,k_r)$ is an index, then the authors (\cite{HMS20}) defined the \emph{polynomial multiple zeta(-star) value} by
\begin{align*}
 \zeta_{x,y}(\boldsymbol{k})&=\sum_{i=0}^{r}\zeta(k_1,\dots,k_i)\zeta(k_r,\dots,k_{i+1})x^{k_1+\dots+k_i}y^{k_{i+1}+\dots+k_r}\in\mathcal{Z}[T][x,y],\\
 \zeta_{x,y}^{\star}(\boldsymbol{k})&=\sum_{i=0}^{r}\zeta^{\star}(k_1,\dots,k_i)\zeta^{\star}(k_r,\dots,k_{i+1})x^{k_1+\dots+k_i}y^{k_{i+1}+\dots+k_r}\in\mathcal{Z}[T][x,y].
\end{align*}
Notice that the polynomial multiple zeta(-star) values are a common generalization of $\zeta^{(\star)}(\boldsymbol{k})$ and $\zeta_S^{(\star)}(\boldsymbol{k})$:
\[
 \zeta_{1,0}(\boldsymbol{k})=\zeta(\boldsymbol{k}),\quad
 \zeta_{1,0}^{\star}(\boldsymbol{k})=\zeta^{\star}(\boldsymbol{k}),\quad
 \zeta_{1,-1}(\boldsymbol{k})=\zeta_S(\boldsymbol{k}),\quad
 \zeta_{1,-1}^{\star}(\boldsymbol{k})=\zeta_S^{\star}(\boldsymbol{k}).
\]

\subsection{Main theorem}
Our main theorem computes the generating functions in Proposition~\ref{prop:gen_func_zeta_fin(k,a,l)} with $\zeta_S$ replaced by $\zeta_{x,y}$.
To state the theorem, we need to define
\[
 \Gamma_1(W)=\exp\Biggl(\sum_{k=1}^{\infty}\frac{\zeta(k)}{k}W^k\Biggr)\in\mathcal{Z}[T][[W]].
\]

\begin{rem}
 Our $\Gamma_1(W)$ is reminiscent of the gamma function, which satisfies
 \[
  \Gamma(z+1)=\exp\Biggl(-\gamma z+\sum_{k=2}^{\infty}\frac{\zeta(k)}{k}(-z)^k\Biggr).
 \]
 Note also that $\exp(-TW)\Gamma_1(W)\in\mathcal{Z}[[W]]$, and that
 \[
  A(W)=\exp(TW)\Gamma_1(-W)=\exp\Biggl(\sum_{k=2}^{\infty}\frac{(-1)^k}{k}\zeta(k)W^k\Biggr)
 \]
 played an essential role in the regularization theorem due to Ihara, Kaneko, and Zagier~\cite{IKZ06}.
\end{rem}

\begin{thm}[Main theorem; Theorem~\ref{thm:gen_func_zeta_xy(k,a,l)}]\label{thm:gen_func_zeta_xy(k,a,l)_intro}
 We have
 \begin{align*}
  &\sum_{\substack{\boldsymbol{k},\boldsymbol{l}\\a\ge2}}\zeta_{x,y}(\boldsymbol{k},a,\boldsymbol{l})A^{\dep\boldsymbol{k}}B^{\dep\boldsymbol{l}}W^{\lvert\boldsymbol{k}\rvert+a+\lvert\boldsymbol{l}\rvert}\\
  &=\frac{yW}{1-B}(\psi_1(y(1-A)W)-\psi_1(y(B-A)W))\frac{\Gamma_1(xW)\Gamma_1(yW)}{\Gamma_1(x(1-A)W)\Gamma_1(y(1-A)W)}\\
  &\hphantom{{}={}}+\frac{xW}{1-A}(\psi_1(x(1-B)W)-\psi_1(x(A-B)W))\frac{\Gamma_1(xW)\Gamma_1(yW)}{\Gamma_1(x(1-B)W)\Gamma_1(y(1-B)W)},\\
  &\sum_{\substack{\boldsymbol{k},\boldsymbol{l}\\a\ge2}}\zeta_{x,y}^{\star}(\boldsymbol{k},a,\boldsymbol{l})A^{\dep\boldsymbol{k}}B^{\dep\boldsymbol{l}}W^{\lvert\boldsymbol{k}\rvert+a+\lvert\boldsymbol{l}\rvert}\\
  &=\frac{yW}{1+A}(\psi_1(y(1+B)W)-\psi_1(y(B-A)W))\frac{\Gamma_1(x(1+A)W)\Gamma_1(y(1+A)W)}{\Gamma_1(xW)\Gamma_1(yW)}\\
  &\hphantom{{}={}}+\frac{xW}{1+B}(\psi_1(x(1+A)W)-\psi_1(x(A-B)W))\frac{\Gamma_1(x(1+B)W)\Gamma_1(y(1+B)W)}{\Gamma_1(xW)\Gamma_1(yW)}
 \end{align*}
 in $\mathcal{Z}[T][x,y][A,B][[W]]$.
\end{thm}

\subsection{Corollaries of our main theorem}
\begin{cor}\label{cor:sum_zeta_xy(k,a,l)}
 If $r$ and $s$ are nonnegative integers and $w$ is an integer with $w\ge r+s+2$, then
 \[
  \sum_{\substack{\lvert\boldsymbol{k}\rvert+a+\lvert\boldsymbol{l}\rvert=w\\\dep\boldsymbol{k}=r,\dep\boldsymbol{l}=s\\a\ge2}}\zeta_{x,y}(\boldsymbol{k},a,\boldsymbol{l}),
  \sum_{\substack{\lvert\boldsymbol{k}\rvert+a+\lvert\boldsymbol{l}\rvert=w\\\dep\boldsymbol{k}=r,\dep\boldsymbol{l}=s\\a\ge2}}\zeta_{x,y}^{\star}(\boldsymbol{k},a,\boldsymbol{l})
  \in\mathbb{Q}[T,\zeta(2),\dots,\zeta(w)][x,y].
 \]
\end{cor}

\begin{proof}
 The corollary follows from Theorem~\ref{thm:gen_func_zeta_xy(k,a,l)_intro} and the observation that the coefficients of $1,W,\dots,W^{w-1}$ in $\psi_1(W)$ and those of $1,W,\dots,W^w$ in $\Gamma_1(W)$ and $\Gamma_1(W)^{-1}$ belong to $\mathbb{Q}[T,\zeta(2),\dots,\zeta(w)][x,y]$.
\end{proof}

\begin{cor}\label{cor:gen_func_zeta(k,a,l)}
 We have
 \begin{align*}
  &\sum_{\substack{\boldsymbol{k},\boldsymbol{l}\\a\ge2}}\zeta(\boldsymbol{k},a,\boldsymbol{l})A^{\dep\boldsymbol{k}}B^{\dep\boldsymbol{l}}W^{\lvert\boldsymbol{k}\rvert+a+\lvert\boldsymbol{l}\rvert}
  =\frac{W}{1-A}(\psi_1((1-B)W)-\psi_1((A-B)W))\frac{\Gamma_1(W)}{\Gamma_1((1-B)W)},\\
  &\sum_{\substack{\boldsymbol{k},\boldsymbol{l}\\a\ge2}}\zeta^{\star}(\boldsymbol{k},a,\boldsymbol{l})A^{\dep\boldsymbol{k}}B^{\dep\boldsymbol{l}}W^{\lvert\boldsymbol{k}\rvert+a+\lvert\boldsymbol{l}\rvert}
  =\frac{W}{1+B}(\psi_1((1+A)W)-\psi_1((A-B)W))\frac{\Gamma_1((1+B)W)}{\Gamma_1(W)}
 \end{align*}
 in $\mathcal{Z}[T][A,B][[W]]$.
\end{cor}

\begin{proof}
 Set $x=1$ and $y=0$ in Theorem~\ref{thm:gen_func_zeta_xy(k,a,l)_intro}, and observe that $\Gamma_1(0)=1$.
\end{proof}

\begin{rem}
 Corollary~\ref{cor:gen_func_zeta(k,a,l)} is a generalization of Proposition~\ref{prop:gen_func_zeta(k,a)} (or equivalently of Theorem~\ref{thm:sum_formula_MZV});
 indeed, setting $B=0$ in Corollary~\ref{cor:gen_func_zeta(k,a,l)} gives Proposition~\ref{prop:gen_func_zeta(k,a)}.
\end{rem}

\begin{cor}
 If $r$ and $s$ are nonnegative integers and $w$ is an integer with $w\ge r+s+2$, then
 \[
  \sum_{\substack{\lvert\boldsymbol{k}\rvert+a+\lvert\boldsymbol{l}\rvert=w\\\dep\boldsymbol{k}=r,\dep\boldsymbol{l}=s\\a\ge2}}\zeta(\boldsymbol{k},a,\boldsymbol{l}),
  \sum_{\substack{\lvert\boldsymbol{k}\rvert+a+\lvert\boldsymbol{l}\rvert=w\\\dep\boldsymbol{k}=r,\dep\boldsymbol{l}=s\\a\ge2}}\zeta^{\star}(\boldsymbol{k},a,\boldsymbol{l})
  \in\mathbb{Q}[T,\zeta(2),\dots,\zeta(w)].
 \]
\end{cor}

\begin{proof}
 Immediate from Corollary~\ref{cor:sum_zeta_xy(k,a,l)} (or Corollary~\ref{cor:gen_func_zeta(k,a,l)}).
\end{proof}

\begin{cor}\label{cor:gen_func_zeta_S(k,a,l)}
 We have
 \begin{align*}
  &\sum_{\substack{\boldsymbol{k},\boldsymbol{l}\\a\ge2}}\zeta_S(\boldsymbol{k},a,\boldsymbol{l})A^{\dep\boldsymbol{k}}B^{\dep\boldsymbol{l}}W^{\lvert\boldsymbol{k}\rvert+a+\lvert\boldsymbol{l}\rvert}\\
  &=-\frac{W}{1-B}(\psi_1(-(1-A)W)-\psi_1((A-B)W))\frac{\pi W}{\sin\pi W}\cdot\frac{\sin\pi(1-A)W}{\pi(1-A)W}\\
  &\hphantom{{}={}}+\frac{W}{1-A}(\psi_1((1-B)W)-\psi_1((A-B)W))\frac{\pi W}{\sin\pi W}\cdot\frac{\sin\pi(1-B)W}{\pi(1-B)W},\\
  &\sum_{\substack{\boldsymbol{k},\boldsymbol{l}\\a\ge2}}\zeta_S^{\star}(\boldsymbol{k},a,\boldsymbol{l})A^{\dep\boldsymbol{k}}B^{\dep\boldsymbol{l}}W^{\lvert\boldsymbol{k}\rvert+a+\lvert\boldsymbol{l}\rvert}\\
  &=-\frac{W}{1+A}(\psi_1(-(1+B)W)-\psi_1((A-B)W))\frac{\sin\pi W}{\pi W}\cdot\frac{\pi(1+A)W}{\sin\pi(1+A)W}\\
  &\hphantom{{}={}}+\frac{W}{1+B}(\psi_1((1+A)W)-\psi_1((A-B)W))\frac{\sin\pi W}{\pi W}\cdot\frac{\pi(1+B)W}{\sin\pi(1+B)W}
 \end{align*}
 in $\mathcal{Z}[A,B][[W]]$.
\end{cor}

\begin{proof}
 Set $x=1$ and $y=-1$ in Theorem~\ref{thm:gen_func_zeta_xy(k,a,l)_intro}, and use the identity $\Gamma_1(W)\Gamma_1(-W)=\pi W/\sin\pi W$, whose proof will be given as Lemma~\ref{lem:Gamma_1(W)Gamma_1(-W)} in Subsection~\ref{subsec:proof_props}.
\end{proof}

\begin{cor}
 If $r$ and $s$ are nonnegative integers and $w$ is an integer with $w\ge r+s+2$, then
 \[
  \sum_{\substack{\lvert\boldsymbol{k}\rvert+a+\lvert\boldsymbol{l}\rvert=w\\\dep\boldsymbol{k}=r,\dep\boldsymbol{l}=s\\a\ge2}}\zeta_S(\boldsymbol{k},a,\boldsymbol{l}),
  \sum_{\substack{\lvert\boldsymbol{k}\rvert+a+\lvert\boldsymbol{l}\rvert=w\\\dep\boldsymbol{k}=r,\dep\boldsymbol{l}=s\\a\ge2}}\zeta_S^{\star}(\boldsymbol{k},a,\boldsymbol{l})
  \in\mathbb{Q}[\zeta(2),\dots,\zeta(w)].
 \]
\end{cor}

\begin{proof}
 Immediate from Corollary~\ref{cor:gen_func_zeta_S(k,a,l)}.
\end{proof}

\subsection{Proofs of propositions and an identity stated in this section}\label{subsec:proof_props}
\begin{proof}[Proof of Proposition~\ref{prop:gen_func_zeta(k,a)}]
 Theorem~\ref{thm:sum_formula_MZV} shows that
 \begin{align*}
  \sum_{\substack{\boldsymbol{k}\\a\ge2}}\zeta(\boldsymbol{k},a)A^{\dep\boldsymbol{k}}W^{\lvert\boldsymbol{k}\rvert+a}
  &=\sum_{\substack{r\ge0\\w\ge r+2}}\sum_{\substack{\lvert\boldsymbol{k}\rvert+a=w\\\dep\boldsymbol{k}=r\\a\ge2}}\zeta(\boldsymbol{k},a)A^rW^w
   =\sum_{\substack{r\ge0\\w\ge r+2}}\zeta(w)A^rW^w\\
  &=\sum_{w=2}^{\infty}\sum_{r=0}^{w-2}A^r\zeta(w)W^w
   =\sum_{w=2}^{\infty}\frac{1-A^{w-1}}{1-A}\zeta(w)W^w\\
  &=\frac{W}{1-A}(\psi_1(W)-\psi_1(AW))
 \end{align*}
 and
 \begin{align*}
  \sum_{\substack{\boldsymbol{k}\\a\ge2}}\zeta^{\star}(\boldsymbol{k},a)A^{\dep\boldsymbol{k}}W^{\lvert\boldsymbol{k}\rvert+a}
  &=\sum_{\substack{r\ge0\\w\ge r+2}}\sum_{\substack{\lvert\boldsymbol{k}\rvert+a=w\\\dep\boldsymbol{k}=r\\a\ge2}}\zeta^{\star}(\boldsymbol{k},a)A^rW^w\\
  &=\sum_{\substack{r\ge0\\w\ge r+2}}\binom{w-1}{r}\zeta(w)A^rW^w\\
  &=\sum_{w=2}^{\infty}\sum_{r=0}^{w-2}\binom{w-1}{r}A^r\zeta(w)W^w\\
  &=\sum_{w=2}^{\infty}((1+A)^{w-1}-A^{w-1})\zeta(w)W^w\\
  &=W(\psi_1((1+A)W)-\psi_1(AW)),
 \end{align*}
 as required.
\end{proof}

\begin{proof}[Proof of Proposition~\ref{prop:gen_func_zeta_fin(k,a,l)}]
 Theorem~\ref{thm:sum_formula_SMZV} shows that
 \begin{align*}
  &\sum_{\substack{\boldsymbol{k},\boldsymbol{l}\\a\ge2}}\zeta_S(\boldsymbol{k},a,\boldsymbol{l})A^{\dep\boldsymbol{k}}B^{\dep\boldsymbol{l}}W^{\lvert\boldsymbol{k}\rvert+a+\lvert\boldsymbol{l}\rvert}\\
  &=\sum_{\substack{r,s\ge0\\w\ge r+s+2}}\sum_{\substack{\lvert\boldsymbol{k}\rvert+a+\lvert\boldsymbol{l}\rvert=w\\\dep\boldsymbol{k}=r,\dep\boldsymbol{l}=s\\a\ge2}}\zeta_S(\boldsymbol{k},a,\boldsymbol{l})A^rB^sW^w\\
  &\equiv\sum_{\substack{r,s\ge0\\w\ge r+s+2}}\biggl(-(-1)^r\binom{w-1}{r}+(-1)^s\binom{w-1}{s}\biggr)\zeta(w)A^rB^sW^w\\
  &=\sum_{w=2}^{\infty}\Biggl(-\sum_{r=0}^{w-2}\sum_{s=0}^{w-r-2}\binom{w-1}{r}(-A)^rB^s+\sum_{s=0}^{w-2}\sum_{r=0}^{w-s-2}\binom{w-1}{s}A^r(-B)^s\Biggr)\zeta(w)W^w\\
  &=\sum_{w=2}^{\infty}\Biggl(-\sum_{r=0}^{w-2}\binom{w-1}{r}(-A)^r\frac{1-B^{w-r-1}}{1-B}+\sum_{s=0}^{w-2}\binom{w-1}{s}\frac{1-A^{w-s-1}}{1-A}(-B)^s\Biggr)\zeta(w)W^w\\
  &=\sum_{w=2}^{\infty}\Biggl(-\frac{(1-A)^{w-1}-(B-A)^{w-1}}{1-B}+\frac{(1-B)^{w-1}-(A-B)^{w-1}}{1-A}\Biggr)\zeta(w)W^w\\
  &=-\frac{W}{1-B}(\psi_1((1-A)W)-\psi_1((B-A)W))+\frac{W}{1-A}(\psi_1((1-B)W)-\psi_1((A-B)W))
 \end{align*}
 and
 \begin{align*}
  &\sum_{\substack{\boldsymbol{k},\boldsymbol{l}\\a\ge2}}\zeta_S^{\star}(\boldsymbol{k},a,\boldsymbol{l})A^{\dep\boldsymbol{k}}B^{\dep\boldsymbol{l}}W^{\lvert\boldsymbol{k}\rvert+a+\lvert\boldsymbol{l}\rvert}\\
  &=\sum_{\substack{r,s\ge0\\w\ge r+s+2}}\sum_{\substack{\lvert\boldsymbol{k}\rvert+a+\lvert\boldsymbol{l}\rvert=w\\\dep\boldsymbol{k}=r,\dep\boldsymbol{l}=s\\a\ge2}}\zeta_S^{\star}(\boldsymbol{k},a,\boldsymbol{l})A^rB^sW^w\\
  &\equiv\sum_{\substack{r,s\ge0\\w\ge r+s+2}}\biggl((-1)^s\binom{w-1}{r}-(-1)^r\binom{w-1}{s}\biggr)\zeta(w)A^rB^sW^w\\
  &=\sum_{w=2}^{\infty}\Biggl(\sum_{r=0}^{w-2}\sum_{s=0}^{w-r-2}\binom{w-1}{r}A^r(-B)^s-\sum_{s=0}^{w-2}\sum_{r=0}^{w-s-2}\binom{w-1}{s}(-A)^rB^s\Biggr)\zeta(w)W^w\\
  &=\sum_{w=2}^{\infty}\Biggl(\sum_{r=0}^{w-2}\binom{w-1}{r}A^r\frac{1-(-B)^{w-r-1}}{1+B}-\sum_{s=0}^{w-2}\binom{w-1}{s}\frac{1-(-A)^{w-s-1}}{1+A}B^s\Biggr)\zeta(w)W^w\\
  &=\sum_{w=2}^{\infty}\Biggl(\frac{(1+A)^{w-1}-(A-B)^{w-1}}{1+B}-\frac{(1+B)^{w-1}-(B-A)^{w-1}}{1+A}\Biggr)\zeta(w)W^w\\
  &=\frac{W}{1+B}(\psi_1((1+A)W)-\psi_1((A-B)W))-\frac{W}{1+A}(\psi_1((1+B)W)-\psi_1((B-A)W)),
 \end{align*}
 as required.
\end{proof}

\begin{lem}\label{lem:Gamma_1(W)Gamma_1(-W)}
 We have
 \[
  \Gamma_1(W)\Gamma_1(-W)=\frac{\pi W}{\sin\pi W}
 \]
 in $\mathcal{Z}[[W]]$.
\end{lem}

\begin{proof}
 Since
 \[
  \log(\Gamma_1(W)\Gamma_1(-W))=\sum_{k=1}^{\infty}\frac{\zeta(k)}{k}(W^k+(-W)^k)=\sum_{k=1}^{\infty}\frac{\zeta(2k)}{k}W^{2k}
 \]
 and
 \[
  \log\frac{\pi W}{\sin\pi W}
  =\log\prod_{m=1}^{\infty}\biggl(1-\frac{W^2}{m^2}\biggr)^{-1}
  =-\sum_{m=1}^{\infty}\log\biggl(1-\frac{W^2}{m^2}\biggr)
  =\sum_{k,m=1}^{\infty}\frac{W^{2k}}{km^{2k}}
  =\sum_{k=1}^{\infty}\frac{\zeta(2k)}{k}W^{2k},
 \]
 the lemma follows.
\end{proof}

\section{Hopf algebra formed by the indices}
We first recall Hoffman's result (\cite{Hof00}) that the indices form a Hopf algebra.
We associate to each index $\boldsymbol{k}=(k_1,\dots,k_r)$ a formal symbol $[\boldsymbol{k}]=[k_1,\dots,k_r]$,
and write $\mathcal{I}$ for the $\mathbb{Q}$-linear space of all formal $\mathbb{Q}$-linear combinations of the symbols $[\boldsymbol{k}]$
(introducing such formal symbols facilitates distinction, for example, between $2(k+l)\in\mathbb{Z}$ and $2[k+l]\in\mathcal{I}$).

For ease of notation, if $\boldsymbol{k}=(k_1,\dots,k_r)$ is an index,
then we write $\boldsymbol{k}_i=(k_1,\dots,k_i)$ and $\boldsymbol{k}^i=(k_{i+1},\dots,k_r)$ for $i=0,\dots,r$,
where we understand that $\boldsymbol{k}_0=\boldsymbol{k}^r=\emptyset$,
and we write $\overleftarrow{\boldsymbol{k}}=(k_r,\dots,k_1)$.

We now define the linear maps that make $\mathcal{I}$ a Hopf algebra.
The multiplication $\mathcal{I}\otimes\mathcal{I}\to\mathcal{I}$, often written as a bilinear product $*$ on $\mathcal{I}$ (known as the \emph{harmonic product} or the \emph{stuffle product}), is defined inductively by setting
\begin{enumerate}
 \item $[\boldsymbol{k}]*[\emptyset]=[\emptyset]*[\boldsymbol{k}]=[\boldsymbol{k}]$ whenever $\boldsymbol{k}$ is an index, and
 \item $[\boldsymbol{k},k]*[\boldsymbol{l},l]=[[\boldsymbol{k},k]*[\boldsymbol{l}],l]+[[\boldsymbol{k}]*[\boldsymbol{l},l],k]+[[\boldsymbol{k}]*[\boldsymbol{l}],k+l]$ whenever $\boldsymbol{k}$ and $\boldsymbol{l}$ are indices and $k$ and $l$ are positive integers, where on the right-hand side we understand that $[\cdot,l]$, $[\cdot,k]$, and $[\cdot,k+l]$ denote the $\mathbb{Q}$-linear operators of concatenating the specified integers.
\end{enumerate}
The unit $\mathbb{Q}\to\mathcal{I}$ is given by $1\mapsto[\emptyset]$.
The comultiplication $\mathcal{I}\to\mathcal{I}\otimes\mathcal{I}$ is defined by
\[
 [\boldsymbol{k}]\mapsto\sum_{i=0}^{r}[\boldsymbol{k}_i]\otimes[\boldsymbol{k}^i]
\]
for indices $\boldsymbol{k}$ of depth $r$.
The counit $\mathcal{I}\to\mathbb{Q}$ is given by
\[
 [\boldsymbol{k}]\mapsto
 \begin{cases}
  1&\text{if $\boldsymbol{k}=\emptyset$};\\
  0&\text{otherwise}
 \end{cases}
\]
for indices $\boldsymbol{k}$.
The antipode $S\colon\mathcal{I}\to\mathcal{I}$ is given by
\[
 S([\boldsymbol{k}])=(-1)^r[\overleftarrow{\boldsymbol{k}}]^{\star},
\]
for indices $\boldsymbol{k}$ of depth $r$.
Here if $\boldsymbol{l}=(l_1,\dots,l_s)$ is an index, then $[\boldsymbol{l}]^{\star}$ denotes the sum of all $[l_1\mathbin{\square}\dotsm\mathbin{\square}l_s]$ with each square replaced by a plus sign or a comma.

\begin{thm}[Hoffman~\cite{Hof00}]\label{thm:I_Hopf_alg}
 The maps given above make $\mathcal{I}$ a commutative Hopf algebra.
\end{thm}

In particular we have the following:
\begin{itemize}
 \item The comultiplication $\mathcal{I}\to\mathcal{I}\otimes\mathcal{I}$ is an algebra homomorphism.
 \item The antipode $S\colon\mathcal{I}\to\mathcal{I}$ is an involution and algebra homomorphism.
  In this paper we find it more convenient to use the $\mathbb{Q}$-linear map $\tilde{S}\colon\mathcal{I}\to\mathcal{I}$ defined by $\tilde{S}([\boldsymbol{k}])=(-1)^r[\boldsymbol{k}]^{\star}$ for indices $\boldsymbol{k}$ of depth $r$;
  it follows that $\tilde{S}$ is also an involution and algebra homomorphism.
 \item If $\boldsymbol{k}$ is an index of depth $r$, then
  \[
   \sum_{i=0}^{r}(-1)^{r-i}[\boldsymbol{k}_i]*[\overleftarrow{\boldsymbol{k}^i}]^{\star}=
   \begin{cases}
    [\emptyset]&\text{if $\boldsymbol{k}=\emptyset$};\\
    0&\text{otherwise}.
   \end{cases}
  \]
\end{itemize}

\section{Generating functions for symmetric sums}
\subsection{Generating functions of $[\boldsymbol{k}]$ and $[\boldsymbol{k}]^{\star}$}
In this subsection, we compute the generating functions
\[
  \sum_{\boldsymbol{k}}[\boldsymbol{k}]A^{\dep\boldsymbol{k}}W^{\lvert\boldsymbol{k}\rvert},\qquad
  \sum_{\boldsymbol{k}}[\boldsymbol{k}]^{\star}A^{\dep\boldsymbol{k}}W^{\lvert\boldsymbol{k}\rvert}
\]
in $\mathcal{I}[A][[W]]$.
To state the results, it is convenient to define the formal power series
\[
 \Gamma_{1,\mathcal{I}}(W)=\exp\Biggl(\sum_{k=1}^{\infty}\frac{[k]}{k}W^k\Biggr)\in\mathcal{I}[[W]].
\]
Observe that
\[
 \tilde{S}(\Gamma_{1,\mathcal{I}}(W))=\exp\Biggl(-\sum_{k=1}^{\infty}\frac{[k]}{k}W^k\Biggr)=\Gamma_{1,\mathcal{I}}(W)^{-1}.
\]

\begin{prop}\label{prop:gen_func_k}
 We have
 \[
  \sum_{\boldsymbol{k}}[\boldsymbol{k}]A^{\dep\boldsymbol{k}}W^{\lvert\boldsymbol{k}\rvert}
  =\frac{\Gamma_{1,\mathcal{I}}(W)}{\Gamma_{1,\mathcal{I}}((1-A)W)},\qquad
  \sum_{\boldsymbol{k}}[\boldsymbol{k}]^{\star}A^{\dep\boldsymbol{k}}W^{\lvert\boldsymbol{k}\rvert}
  =\frac{\Gamma_{1,\mathcal{I}}((1+A)W)}{\Gamma_{1,\mathcal{I}}(W)}
 \]
 in $\mathcal{I}[A][[W]]$.
\end{prop}

\begin{proof}
 The first identity implies the second because
 \begin{align*}
  \sum_{\boldsymbol{k}}[\boldsymbol{k}]^{\star}A^{\dep\boldsymbol{k}}W^{\lvert\boldsymbol{k}\rvert}
  &=\tilde{S}\Biggl(\sum_{\boldsymbol{k}}[\boldsymbol{k}](-A)^{\dep\boldsymbol{k}}W^{\lvert\boldsymbol{k}\rvert}\Biggr)
  =\tilde{S}\biggl(\frac{\Gamma_{1,\mathcal{I}}(W)}{\Gamma_{1,\mathcal{I}}((1+A)W)}\Biggr)\\
  &=\frac{\Gamma_{1,\mathcal{I}}((1+A)W)}{\Gamma_{1,\mathcal{I}}(W)}.
 \end{align*}
 The first identity is equivalent to
 \[
  \log\Biggl(\sum_{\boldsymbol{k}}[\boldsymbol{k}]A^{\dep\boldsymbol{k}}W^{\lvert\boldsymbol{k}\rvert}\Biggr)
  =\sum_{k=1}^{\infty}\frac{[k]}{k}(1-(1-A)^k)W^k,
 \]
 and since both sides have constant term $0$ (with respect to $W$), it suffices to prove that both sides have the same derivative (with respect to $W$):
 \[
  \frac{\sum_{\boldsymbol{k}\ne\emptyset}\lvert\boldsymbol{k}\rvert[\boldsymbol{k}]A^{\dep\boldsymbol{k}}W^{\lvert\boldsymbol{k}\rvert-1}}{\sum_{\boldsymbol{k}}[\boldsymbol{k}]A^{\dep\boldsymbol{k}}W^{\lvert\boldsymbol{k}\rvert}}
  =\sum_{k=1}^{\infty}[k](1-(1-A)^k)W^{k-1},
 \]
 which in turn is equivalent to
 \[
  \Biggl(\sum_{\boldsymbol{k}}[\boldsymbol{k}]A^{\dep\boldsymbol{k}}W^{\lvert\boldsymbol{k}\rvert}\Biggr)\Biggl(\sum_{k=1}^{\infty}[k](1-(1-A)^k)W^k\Biggr)=\sum_{\boldsymbol{k}\ne\emptyset}\lvert\boldsymbol{k}\rvert[\boldsymbol{k}]A^{\dep\boldsymbol{k}}W^{\lvert\boldsymbol{k}\rvert}.
 \]

 For each nonempty index $\boldsymbol{l}=(l_1,\dots,l_s)$,
 the coefficient of $[\boldsymbol{l}]W^{\lvert\boldsymbol{l}\rvert}$ in the left-hand side is
 \[
  \sum_{j=1}^{s}A^{s-1}(1-(1-A)^{l_j})+\sum_{j=1}^{s}A^s\sum_{i=1}^{l_j-1}(1-(1-A)^i),
 \]
 which simplifies to $A^s\sum_{j=1}^{s}l_j=\lvert\boldsymbol{l}\rvert A^{\dep\boldsymbol{l}}$.
\end{proof}

\begin{rem}\label{rem:Gamma_1I_expansion}
 Substituting $A=1$ and $A=-1$ into the equations in Proposition~\ref{prop:gen_func_k} respectively gives
 \begin{gather*}
  \Gamma_{1,\mathcal{I}}(W)=\sum_{\boldsymbol{k}}[\boldsymbol{k}]W^{\lvert\boldsymbol{k}\rvert}=\sum_{k=0}^{\infty}[\{1\}^k]^{\star}W^k,\\
  \Gamma_{1,\mathcal{I}}(W)^{-1}=\sum_{\boldsymbol{k}}(-1)^{\dep\boldsymbol{k}}[\boldsymbol{k}]^{\star}W^{\lvert\boldsymbol{k}\rvert}=\sum_{k=0}^{\infty}(-1)^k[\{1\}^k]W^k,
 \end{gather*}
 where $\{1\}^k$ denotes the index $(\underbrace{1,\dots,1}_{k})$, which means $\emptyset$ if $k=0$.
\end{rem}

\subsection{Generating functions for symmetric sums of $[\boldsymbol{k}]_{x,y}$ and $[\boldsymbol{k}]_{x,y}^{\star}$}
If $\boldsymbol{k}$ is an index, then we define
\begin{align*}
 [\boldsymbol{k}]_{x,y}&=\sum_{i=0}^{r}[\boldsymbol{k}_i]*[\overleftarrow{\boldsymbol{k}^i}]x^{\lvert\boldsymbol{k}_i\rvert}y^{\lvert\boldsymbol{k}^i\rvert}\in\mathcal{I}[x,y],\\
 [\boldsymbol{k}]_{x,y}^{\star}&=\sum_{i=0}^{r}[\boldsymbol{k}_i]^{\star}*[\overleftarrow{\boldsymbol{k}^i}]^{\star}x^{\lvert\boldsymbol{k}_i\rvert}y^{\lvert\boldsymbol{k}^i\rvert}\in\mathcal{I}[x,y],
\end{align*}
where $r=\dep\boldsymbol{k}$.
Note that if $\boldsymbol{k}$ is an index of depth $r$, then
\begin{align*}
 \tilde{S}([\boldsymbol{k}]_{x,y})
 &=\sum_{i=0}^{r}\tilde{S}([\boldsymbol{k}_i])*\tilde{S}([\overleftarrow{\boldsymbol{k}^i}])x^{\lvert\boldsymbol{k}_i\rvert}y^{\lvert\boldsymbol{k}^i\rvert}\\
 &=\sum_{i=0}^{r}(-1)^i[\boldsymbol{k}_i]^{\star}*(-1)^{r-i}[\overleftarrow{\boldsymbol{k}^i}]^{\star}x^{\lvert\boldsymbol{k}_i\rvert}y^{\lvert\boldsymbol{k}^i\rvert}\\
 &=(-1)^r[\boldsymbol{k}]_{x,y}^{\star}.
\end{align*}

\begin{lem}
 The $\mathbb{Q}$-linear map from $\mathcal{I}$ to $\mathcal{I}[x,y]$ given by
 $[\boldsymbol{k}]\mapsto[\boldsymbol{k}]_{x,y}$ for indices $\boldsymbol{k}$ is an algebra homomorphism.
\end{lem}

\begin{proof}
 The map in question is the composite
 \begin{align*}
  \mathcal{I}&\to\mathcal{I}\otimes\mathcal{I}\\
  &\to\mathcal{I}[x]\otimes\mathcal{I}[y]\cong\mathcal{I}\otimes\mathbb{Q}[x]\otimes\mathcal{I}\otimes\mathbb{Q}[y]\cong\mathcal{I}\otimes\mathcal{I}\otimes\mathbb{Q}[x,y]\\
  &\to\mathcal{I}\otimes\mathbb{Q}[x,y]\cong\mathcal{I}[x,y],
 \end{align*}
 where the arrows denote the comultiplication, 
 the map $[\boldsymbol{k}]\otimes[\boldsymbol{l}]\mapsto[\boldsymbol{k}]x^{\lvert\boldsymbol{k}\rvert}\otimes[\overleftarrow{\boldsymbol{l}}]y^{\lvert\boldsymbol{l}\rvert}$, and the multiplication.
\end{proof}

\begin{prop}\label{prop:gen_func_kxy}
 We have
 \begin{align*}
  \sum_{\boldsymbol{k}}[\boldsymbol{k}]_{x,y}A^{\dep\boldsymbol{k}}W^{\lvert\boldsymbol{k}\rvert}
  &=\frac{\Gamma_{1,\mathcal{I}}(xW)\Gamma_{1,\mathcal{I}}(yW)}{\Gamma_{1,\mathcal{I}}(x(1-A)W)\Gamma_{1,\mathcal{I}}(y(1-A)W)},\\
  \sum_{\boldsymbol{k}}[\boldsymbol{k}]_{x,y}^{\star}A^{\dep\boldsymbol{k}}W^{\lvert\boldsymbol{k}\rvert}
  &=\frac{\Gamma_{1,\mathcal{I}}(x(1+A)W)\Gamma_{1,\mathcal{I}}(y(1+A)W)}{\Gamma_{1,\mathcal{I}}(xW)\Gamma_{1,\mathcal{I}}(yW)}
 \end{align*}
 in $\mathcal{I}[x,y][A][[W]]$.
\end{prop}

\begin{proof}
 The first identity implies the second because
 \begin{align*}
  \sum_{\boldsymbol{k}}[\boldsymbol{k}]_{x,y}^{\star}A^{\dep\boldsymbol{k}}W^{\lvert\boldsymbol{k}\rvert}
  &=\tilde{S}\Biggl(\sum_{\boldsymbol{k}}[\boldsymbol{k}]_{x,y}(-A)^{\dep\boldsymbol{k}}W^{\lvert\boldsymbol{k}\rvert}\Biggr)\\
  &=\tilde{S}\biggl(\frac{\Gamma_{1,\mathcal{I}}(xW)\Gamma_{1,\mathcal{I}}(yW)}{\Gamma_{1,\mathcal{I}}(x(1+A)W)\Gamma_{1,\mathcal{I}}(y(1+A)W)}\biggr)\\
  &=\frac{\Gamma_{1,\mathcal{I}}(x(1+A)W)\Gamma_{1,\mathcal{I}}(y(1+A)W)}{\Gamma_{1,\mathcal{I}}(xW)\Gamma_{1,\mathcal{I}}(yW)}.
 \end{align*}
 Since the algebra homomorphism $[\boldsymbol{k}]\mapsto[\boldsymbol{k}]_{x,y}$ satisfies
 \[
  \Gamma_{1,\mathcal{I}}(W)
  \mapsto\exp\Biggl(\sum_{k=1}^{\infty}\frac{[k]_{x,y}}{k}W^k\Biggr)
  =\exp\Biggl(\sum_{k=1}^{\infty}\frac{[k](x^k+y^k)}{k}W^k\Biggr)
  =\Gamma_{1,\mathcal{I}}(xW)\Gamma_{1,\mathcal{I}}(yW),
 \]
 the first identity follows from Proposition~\ref{prop:gen_func_k}.
\end{proof}

\subsection{Generating functions of $\zeta^{(\star)}(\boldsymbol{k})$, $\zeta_{x,y}^{(\star)}(\boldsymbol{k})$, and $\zeta_S^{(\star)}(\boldsymbol{k})$}
We define $\mathbb{Q}$-linear maps $Z\colon\mathcal{I}\to\mathcal{Z}[T]$, $Z_S\colon\mathcal{I}\to\mathcal{Z}$, and $Z_{x,y}\colon\mathcal{I}\to\mathcal{Z}[T][x,y]$ by setting
$Z([\boldsymbol{k}])=\zeta(\boldsymbol{k})$, $Z_S([\boldsymbol{k}])=\zeta_S(\boldsymbol{k})=Z([\boldsymbol{k}]_{1,-1})$, and $Z_{x,y}([\boldsymbol{k}])=\zeta_{x,y}(\boldsymbol{k})=Z([\boldsymbol{k}]_{x,y})$.
Then they are all algebra homomorphisms, and satisfy
$Z([\boldsymbol{k}]^{\star})=\zeta^{\star}(\boldsymbol{k})$, $Z_S([\boldsymbol{k}]^{\star})=\zeta_S^{\star}(\boldsymbol{k})$, and $Z_{x,y}([\boldsymbol{k}]^{\star})=\zeta_{x,y}^{\star}(\boldsymbol{k})$.

We have $Z(\Gamma_{1,\mathcal{I}}(W))=\Gamma_1(W)$, and
Remark~\ref{rem:Gamma_1I_expansion} shows that
\begin{gather*}
 \Gamma_1(W)=\exp\Biggl(\sum_{k=1}^{\infty}\frac{\zeta(k)}{k}W^k\Biggr)=\sum_{\boldsymbol{k}}\zeta(\boldsymbol{k})W^{\lvert\boldsymbol{k}\rvert}=\sum_{k=0}^{\infty}\zeta^{\star}(\{1\}^k)W^k,\\
 \Gamma_1(W)^{-1}=\exp\Biggl(-\sum_{k=1}^{\infty}\frac{\zeta(k)}{k}W^k\Biggr)=\sum_{\boldsymbol{k}}(-1)^{\dep\boldsymbol{k}}\zeta^{\star}(\boldsymbol{k})W^{\lvert\boldsymbol{k}\rvert}=\sum_{k=0}^{\infty}(-1)^k\zeta(\{1\}^k)W^k.
\end{gather*}

\begin{prop}\label{prop:gen_func_zeta(k)}
 We have
 \[
  \sum_{\boldsymbol{k}}\zeta(\boldsymbol{k})A^{\dep\boldsymbol{k}}W^{\lvert\boldsymbol{k}\rvert}
  =\frac{\Gamma_1(W)}{\Gamma_1((1-A)W)},\qquad
  \sum_{\boldsymbol{k}}\zeta^{\star}(\boldsymbol{k})A^{\dep\boldsymbol{k}}W^{\lvert\boldsymbol{k}\rvert}
  =\frac{\Gamma_1((1+A)W)}{\Gamma_1(W)}
 \]
 in $\mathcal{Z}[T][A][[W]]$.
\end{prop}

\begin{proof}
 Immediate from Proposition~\ref{prop:gen_func_k}.
\end{proof}

\begin{prop}\label{prop:gen_func_zeta_xy(k)}
 We have
 \begin{align*}
  \sum_{\boldsymbol{k}}\zeta_{x,y}(\boldsymbol{k})A^{\dep\boldsymbol{k}}W^{\lvert\boldsymbol{k}\rvert}
  &=\frac{\Gamma_1(xW)\Gamma_1(yW)}{\Gamma_1(x(1-A)W)\Gamma_1(y(1-A)W)},\\
  \sum_{\boldsymbol{k}}\zeta_{x,y}^{\star}(\boldsymbol{k})A^{\dep\boldsymbol{k}}W^{\lvert\boldsymbol{k}\rvert}
  &=\frac{\Gamma_1(x(1+A)W)\Gamma_1(y(1+A)W)}{\Gamma_1(xW)\Gamma_1(yW)}
 \end{align*}
 in $\mathcal{Z}[T][x,y][A][[W]]$.
\end{prop}

\begin{proof}
 Immediate from Proposition~\ref{prop:gen_func_kxy}.
\end{proof}

\begin{cor}\label{cor:sum_zeta_xy_Q_riemann}
 If $r$ is a nonnegative integer and $w$ is an integer with $w\ge r$, then
 \[
  \sum_{\substack{\lvert\boldsymbol{k}\rvert=w\\\dep\boldsymbol{k}=r}}\zeta_{x,y}(\boldsymbol{k}),
  \sum_{\substack{\lvert\boldsymbol{k}\rvert=w\\\dep\boldsymbol{k}=r}}\zeta_{x,y}^{\star}(\boldsymbol{k})
  \in\mathbb{Q}[T,\zeta(2),\dots,\zeta(w)][x,y].
 \]
\end{cor}

\begin{proof}
 Immediate from Proposition~\ref{prop:gen_func_zeta_xy(k)}.
\end{proof}

\begin{prop}\label{prop:gen_func_zeta_S(k)}
 We have
 \begin{align*}
  \sum_{\boldsymbol{k}}\zeta_S(\boldsymbol{k})A^{\dep\boldsymbol{k}}W^{\lvert\boldsymbol{k}\rvert}
  &=\frac{\pi W}{\sin\pi W}\cdot\frac{\sin\pi(1-A)W}{\pi(1-A)W},\\
  \sum_{\boldsymbol{k}}\zeta_S^{\star}(\boldsymbol{k})A^{\dep\boldsymbol{k}}W^{\lvert\boldsymbol{k}\rvert}
  &=\frac{\sin\pi W}{\pi W}\cdot\frac{\pi(1+A)W}{\sin\pi(1+A)W}
 \end{align*}
 in $\mathcal{Z}[A][[W]]$.
\end{prop}

\begin{proof}
 Set $x=1$ and $y=-1$ in Proposition~\ref{prop:gen_func_zeta_xy(k)} and use Lemma~\ref{lem:Gamma_1(W)Gamma_1(-W)}.
\end{proof}

\begin{cor}
 If $r$ is a nonnegative integer and $w$ is an integer with $w\ge r$, then
 \[
  \sum_{\substack{\lvert\boldsymbol{k}\rvert=k\\\dep\boldsymbol{k}=r}}\zeta_S(\boldsymbol{k}),
  \sum_{\substack{\lvert\boldsymbol{k}\rvert=k\\\dep\boldsymbol{k}=r}}\zeta_S^{\star}(\boldsymbol{k})
  \begin{cases}
   =0&\text{if $w$ is odd};\\
   \in\mathbb{Q}\pi^w&\text{if $w$ is even}.
  \end{cases}
 \]
\end{cor}

\begin{proof}
 Since
 \[
  \frac{\sin\pi W}{\pi W},\frac{\pi W}{\sin\pi W}\in\mathbb{Q}[\pi^2W^2],
 \]
 the corollary is immediate from Proposition~\ref{prop:gen_func_zeta_S(k)}.
\end{proof}

\section{Schur multiple zeta values of anti-hook type}
\subsection{Schur multiple zeta values of anti-hook type}
When investigating the relationship between the sum formulas for multiple zeta(-star) values and symmetric multiple zeta(-star) values, we find it necessary to use the Schur multiple zeta values
(defined by Nakasuji, Phuksuwan, and Yamasaki \cite{NPY18}) of anti-hook type.
If $(k_1,\dots,k_r)$ and $(l_1,\dots,l_s)$ are indices and $a$ is a positive integer with $a\ge2$, then the \emph{Schur multiple zeta value of anti-hook type} is defined as
\[
 \zeta\left(\;\ytableausetup{centertableaux}\begin{ytableau}\none&\none&\none&k_1\\\none&\none&\none&\vdots\\\none&\none&\none&k_r\\l_1&\cdots&l_s&a\end{ytableau}\;\right)
 =\sum\frac{1}{m_1^{k_1}\dotsm m_r^{k_r}n_1^{l_1}\dotsm n_s^{l_s}p^a},
\]
where the sum runs over all positive integers $m_1,\dots,m_r,n_1,\dots,n_s,p$ satisfying $m_1<\dots<m_r<p$ and $n_1\le\dots\le n_s\le p$.
To save space, we write $\left(\begin{array}{c}k_1,\dots,k_r\\l_1,\dots,l_s\end{array};a\right)$ for what lies between the parentheses in the left-hand side.
The Schur multiple zeta values of anti-hook type are a common generalization of the multiple zeta and zeta-star values:
\[
 \zeta\left(\begin{array}{c}\boldsymbol{k}\\\emptyset\end{array};a\right)=\zeta(\boldsymbol{k},a),\qquad
 \zeta\left(\begin{array}{c}\emptyset\\\boldsymbol{l}\end{array};a\right)=\zeta^{\star}(\boldsymbol{l},a)
\]
if $\boldsymbol{k}$ and $\boldsymbol{l}$ are indices and $a\ge2$.
An advantage of using the Schur multiple zeta values of anti-hook type is that
they allow us to succinctly express the harmonic product of the multiple zeta value and the multiple zeta-star value:
\begin{align*}
 \zeta(k_1,\dots,k_r)\zeta^{\star}(l_1,\dots,l_s)
 &=\Biggl(\sum_{1\le m_1<\dots<m_r}\frac{1}{m_1^{k_1}\dotsm m_r^{k_r}}\Biggr)\Biggl(\sum_{1\le n_1\le\dots\le n_s}\frac{1}{n_1^{l_1}\dotsm n_s^{l_s}}\Biggr)\\
 &=\Biggl(\sum_{\substack{1\le m_1<\dots<m_r\\1\le n_1\le\dots\le n_s\\m_r\ge n_s}}+\sum_{\substack{1\le m_1<\dots<m_r\\1\le n_1\le\dots\le n_s\\m_r<n_s}}\Biggr)\frac{1}{m_1^{k_1}\dotsm m_r^{k_r}n_1^{l_1}\dotsm n_s^{l_s}}\\
 &=\zeta\left(\begin{array}{c}k_1,\dots,k_{r-1}\\l_1,\dots,l_s\end{array};k_r\right)+\zeta\left(\begin{array}{c}k_1,\dots,k_r\\l_1,\dots,l_{s-1}\end{array};l_s\right)
\end{align*}
if $(k_1,\dots,k_r)$ and $(l_1,\dots,l_s)$ are nonempty admissible indices.
Observe that the Schur multiple zeta value can always be written as a $\mathbb{Q}$-linear combination of multiple zeta values; for example we have
\[
 \zeta\left(\begin{array}{c}k\\l\end{array};a\right)
 =\zeta(k,l,a)+\zeta(l,k,a)+\zeta(k+l,a)+\zeta(k,l+a).
\]

\subsection{Elements of $\mathcal{I}$ corresponding to Schur multiple zeta values of anti-hook type}
The observations made in the previous subsection lead us to the following formal definition of $\left[\begin{array}{c}k_1,\dots,k_r\\l_1,\dots,l_s\end{array};a\right]\in\mathcal{I}$:
\begin{defn}\label{defn:Schur}
 We define
 \[
  \left[\begin{array}{c}k_1,\dots,k_r\\l_1,\dots,l_s\end{array};a\right]\in\mathcal{I}
 \]
 for each pair of indices $(k_1,\dots,k_r)$ and $(l_1,\dots,l_s)$ and each positive integer $a$
 so that the following properties are fulfilled:
 \begin{enumerate}
  \item $\left[\begin{array}{c}k_1,\dots,k_r\\\emptyset\end{array};a\right]=[k_1,\dots,k_r,a]$ and $\left[\begin{array}{c}\emptyset\\l_1,\dots,l_s\end{array};a\right]=[l_1,\dots,l_s,a]^{\star}$
   whenever $(k_1,\dots,k_r)$ and $(l_1,\dots,l_s)$ are indices and $a$ is a positive integer;
  \item $\left[\begin{array}{c}k_1,\dots,k_{r-1}\\l_1,\dots,l_s\end{array};k_r\right]+\left[\begin{array}{c}k_1,\dots,k_r\\l_1,\dots,l_{s-1}\end{array};l_s\right]=[k_1,\dots,k_r]*[l_1,\dots,l_s]^{\star}$
   whenever $(k_1,\dots,k_r)$ and $(l_1,\dots,l_s)$ are nonempty indices.
 \end{enumerate}
\end{defn}

Observe that the definition above does indeed uniquely determine
$\left[\begin{array}{c}k_1,\dots,k_r\\l_1,\dots,l_s\end{array};a\right]\in\mathcal{I}$
as the following example illustrates.
The required properties imply that
\begin{align*}
 \left[\begin{array}{c}k_1,k_2,k_3\\\emptyset\end{array};k_4\right]\phantom{+\left[\begin{array}{c}k_1,k_2\\k_4\end{array};k_3\right]+\left[\begin{array}{c}k_1\\k_4,k_3\end{array};k_2\right]+\left[\begin{array}{c}\emptyset\\k_4,k_3,k_2\end{array};k_1\right]}&=[k_1,k_2,k_3,k_4],\\
 \left[\begin{array}{c}k_1,k_2,k_3\\\emptyset\end{array};k_4\right]+\left[\begin{array}{c}k_1,k_2\\k_4\end{array};k_3\right]\phantom{+\left[\begin{array}{c}k_1\\k_4,k_3\end{array};k_2\right]+\left[\begin{array}{c}\emptyset\\k_4,k_3,k_2\end{array};k_1\right]}&=[k_1,k_2,k_3]*[k_4]^{\star},\\
 \phantom{\left[\begin{array}{c}k_1,k_2,k_3\\\emptyset\end{array};k_4\right]+{}}\left[\begin{array}{c}k_1,k_2\\k_4\end{array};k_3\right]+\left[\begin{array}{c}k_1\\k_4,k_3\end{array};k_2\right]\phantom{+\left[\begin{array}{c}\emptyset\\k_4,k_3,k_2\end{array};k_1\right]}&=[k_1,k_2]*[k_4,k_3]^{\star},\\
 \phantom{\left[\begin{array}{c}k_1,k_2,k_3\\\emptyset\end{array};k_4\right]+\left[\begin{array}{c}k_1,k_2\\k_4\end{array};k_3\right]+{}}\left[\begin{array}{c}k_1\\k_4,k_3\end{array};k_2\right]+\left[\begin{array}{c}\emptyset\\k_4,k_3,k_2\end{array};k_1\right]&=[k_1]*[k_4,k_3,k_2]^{\star},\\
 \phantom{\left[\begin{array}{c}k_1,k_2,k_3\\\emptyset\end{array};k_4\right]+\left[\begin{array}{c}k_1,k_2\\k_4\end{array};k_3\right]+\left[\begin{array}{c}k_1\\k_4,k_3\end{array};k_2\right]+{}}\left[\begin{array}{c}\emptyset\\k_4,k_3,k_2\end{array};k_1\right]&=[k_4,k_3,k_2,k_1]^{\star}.
\end{align*}
Although the requirements are superfluous (there are $4$ unknowns and $5$ equations in the example above), the third remark after Theorem~\ref{thm:I_Hopf_alg} shows that they are compatible.

\begin{prop}
 If $(k_1,\dots,k_r)$ and $(l_1,\dots,l_s)$ are indices and $a$ is a positive integer, then we have
 \[
  \tilde{S}\biggl(\left[\begin{array}{c}k_1,\dots,k_r\\l_1,\dots,l_s\end{array};a\right]\biggr)=(-1)^{r+s+1}\left[\begin{array}{c}l_1,\dots,l_s\\k_1,\dots,k_r\end{array};a\right].
 \]
\end{prop}

\begin{proof}
 Since $\tilde{S}$ is an involution, the proposition is equivalent to showing that
 \[
  \left[\begin{array}{c}k_1,\dots,k_r\\l_1,\dots,l_s\end{array};a\right]=(-1)^{r+s+1}\tilde{S}\biggl(\left[\begin{array}{c}l_1,\dots,l_s\\k_1,\dots,k_r\end{array};a\right]\biggr).
 \]
 The map $((k_1,\dots,k_r),(l_1,\dots,l_s),a)\mapsto(-1)^{r+s+1}\tilde{S}\biggl(\left[\begin{array}{c}l_1,\dots,l_s\\k_1,\dots,k_r\end{array};a\right]\biggr)$ satisfies the properties in Definition~\ref{defn:Schur} because
 \begin{align*}
  (-1)^{r+1}\tilde{S}\biggl(\left[\begin{array}{c}\emptyset\\k_1,\dots,k_r\end{array};a\right]\biggr)=(-1)^{r+1}\tilde{S}([k_1,\dots,k_r,a]^{\star})=[k_1,\dots,k_r,a],\\
  (-1)^{s+1}\tilde{S}\biggl(\left[\begin{array}{c}l_1,\dots,l_s\\\emptyset\end{array};a\right]\biggr)=(-1)^{s+1}\tilde{S}([l_1,\dots,l_s,a])=[l_1,\dots,l_s,a]^{\star}
 \end{align*}
 and
 \begin{align*}
  &(-1)^{r+s}\tilde{S}\biggl(\left[\begin{array}{c}l_1,\dots,l_s\\k_1,\dots,k_{r-1}\end{array};k_r\right]\biggr)+(-1)^{r+s}\tilde{S}\biggl(\left[\begin{array}{c}l_1,\dots,l_{s-1}\\k_1,\dots,k_r\end{array};l_s\right]\biggr)\\
  &=(-1)^{r+s}\tilde{S}\biggl(\left[\begin{array}{c}l_1,\dots,l_s\\k_1,\dots,k_{r-1}\end{array};k_r\right]+\left[\begin{array}{c}l_1,\dots,l_{s-1}\\k_1,\dots,k_r\end{array};l_s\right]\biggr)\\
  &=(-1)^{r+s}\tilde{S}([l_1,\dots,l_s]*[k_1,\dots,k_r]^{\star})\\
  &=(-1)^{r+s}\tilde{S}([l_1,\dots,l_s])*\tilde{S}([k_1,\dots,k_r]^{\star})\\
  &=[l_1,\dots,l_s]^{\star}*[k_1,\dots,k_r],
 \end{align*}
 from which the proposition follows.
\end{proof}

\begin{lem}\label{lem:alternating2}
 If $\boldsymbol{k}$ and $\boldsymbol{l}$ are indices and $a$ is a positive integer, then we have
 \[
  \sum_{j=0}^{s}(-1)^j[\boldsymbol{k},a,\boldsymbol{l}_j]*[\overleftarrow{\boldsymbol{l}^j}]^{\star}=\left[\begin{array}{c}\boldsymbol{k}\\\overleftarrow{\boldsymbol{l}}\end{array};a\right],
 \]
 where $s=\dep\boldsymbol{l}$.
\end{lem}

\begin{proof}
 If we write $\boldsymbol{l}=(l_1,\dots,l_s)$, then Definition~\ref{defn:Schur} implies that
 \[
  [\boldsymbol{k},a,\boldsymbol{l}_j]*[\overleftarrow{\boldsymbol{l}^j}]^{\star}=\left[\begin{array}{c}\boldsymbol{k},a,\boldsymbol{l}_{j-1}\\\overleftarrow{\boldsymbol{l}^j}\end{array};l_j\right]+\left[\begin{array}{c}\boldsymbol{k},a,\boldsymbol{l}_j\\\overleftarrow{\boldsymbol{l}^{j+1}}\end{array};l_{j+1}\right]
 \]
 for $j=0,\dots,s$, where in the right-hand side we understand that the first term is $\left[\begin{array}{c}\boldsymbol{k}\\\overleftarrow{\boldsymbol{l}}\end{array};a\right]$ if $j=0$ and that the second term is $0$ if $j=s$.
 This immediately implies the lemma.
\end{proof}

If $\boldsymbol{k}=(k_1,\dots,k_r)$ is an index and $i$ and $i'$ are integers with $0\le i\le i'\le r$,
then we write $\boldsymbol{k}^i_{i'}=(k_{i+1},\dots,k_{i'})$.

\begin{lem}\label{lem:alternating3}
 If $\boldsymbol{k}$ and $\boldsymbol{l}$ are indices and $a$ is a positive integer, then we have
 \[
  \sum_{j=0}^{s}(-1)^j\left[\begin{array}{c}\boldsymbol{k}\\\overleftarrow{\boldsymbol{l}_j}\end{array};a\right]*[\boldsymbol{l}^j]=[\boldsymbol{k},a,\boldsymbol{l}],
 \]
 where $s=\dep\boldsymbol{l}$.
\end{lem}

\begin{proof}
 Lemma~\ref{lem:alternating2} shows that
 \begin{align*}
  \sum_{j=0}^{s}(-1)^j\left[\begin{array}{c}\boldsymbol{k}\\\overleftarrow{\boldsymbol{l}_j}\end{array};a\right]*[\boldsymbol{l}^j]
  &=\sum_{j=0}^{s}(-1)^j\Biggl(\sum_{j'=0}^{j}(-1)^{j'}[\boldsymbol{k},a,\boldsymbol{l}_{j'}]*[\overleftarrow{\boldsymbol{l}^{j'}_j}]^{\star}\Biggr)*[\boldsymbol{l}^j]\\
  &=\sum_{j'=0}^{s}(-1)^{s-j'}[\boldsymbol{k},a,\boldsymbol{l}_{j'}]*\Biggl(\sum_{j=j'}^{s}(-1)^{s-j}[\overleftarrow{\boldsymbol{l}^{j'}_j}]^{\star}*[\boldsymbol{l}^j]\Biggr)\\
  &=[\boldsymbol{k},a,\boldsymbol{l}],
 \end{align*}
 which completes the proof.
\end{proof}

The following lemma will be the key to explaining the relationship between the sum formulas for multiple zeta(-star) values and symmetric multiple zeta(-star) values and also to proving our main theorem:
\begin{lem}\label{lem:key}
 If $\boldsymbol{k}$ and $\boldsymbol{l}$ are indices and $a$ is a positive integer, then we have
 \begin{align*}
  [\boldsymbol{k},a,\boldsymbol{l}]_{x,y}
  &=\sum_{i=0}^{r}(-1)^{r-i}\left[\begin{array}{c}\overleftarrow{\boldsymbol{l}}\\\boldsymbol{k}^i\end{array};a\right]y^{\lvert\boldsymbol{k}^i\rvert+a+\lvert\boldsymbol{l}\rvert}*[\boldsymbol{k}_i]_{x,y}
    +\sum_{j=0}^{s}(-1)^j\left[\begin{array}{c}\boldsymbol{k}\\\overleftarrow{\boldsymbol{l}_j}\end{array};a\right]x^{\lvert\boldsymbol{k}\rvert+a+\lvert\boldsymbol{l}_j\rvert}*[\boldsymbol{l}^j]_{x,y},\\
  [\boldsymbol{k},a,\boldsymbol{l}]_{x,y}^{\star}
  &=\sum_{i=0}^{r}(-1)^{r-i}\left[\begin{array}{c}\boldsymbol{k}^i\\\overleftarrow{\boldsymbol{l}}\end{array};a\right]y^{\lvert\boldsymbol{k}^i\rvert+a+\lvert\boldsymbol{l}\rvert}*[\boldsymbol{k}_i]_{x,y}^{\star}
    +\sum_{j=0}^{s}(-1)^j\left[\begin{array}{c}\overleftarrow{\boldsymbol{l}_j}\\\boldsymbol{k}\end{array};a\right]x^{\lvert\boldsymbol{k}\rvert+a+\lvert\boldsymbol{l}_j\rvert}*[\boldsymbol{l}^j]_{x,y}^{\star},
 \end{align*}
 where $r=\dep\boldsymbol{k}$ and $s=\dep\boldsymbol{l}$.
\end{lem}

\begin{proof}
 The first identity implies the second because
 \begin{align*}
  [\boldsymbol{k},a,\boldsymbol{l}]_{x,y}^{\star}
  &=(-1)^{r+s+1}\tilde{S}([\boldsymbol{k},a,\boldsymbol{l}]_{x,y})\\
  &=(-1)^{r+s+1}\sum_{i=0}^{r}(-1)^{r-i}\tilde{S}\biggl(\left[\begin{array}{c}\overleftarrow{\boldsymbol{l}}\\\boldsymbol{k}^i\end{array};a\right]\biggr)y^{\lvert\boldsymbol{k}^i\rvert+a+\lvert\boldsymbol{l}\rvert}*\tilde{S}([\boldsymbol{k}_i]_{x,y})\\
  &\phantom{{}={}}+(-1)^{r+s+1}\sum_{j=0}^{s}(-1)^j\tilde{S}\biggl(\left[\begin{array}{c}\boldsymbol{k}\\\overleftarrow{\boldsymbol{l}_j}\end{array};a\right]\biggr)x^{\lvert\boldsymbol{k}\rvert+a+\lvert\boldsymbol{l}_j\rvert}*\tilde{S}([\boldsymbol{l}^j]_{x,y})\\
  &=(-1)^{r+s+1}\sum_{i=0}^{r}(-1)^{r-i}(-1)^{r-i+s+1}\left[\begin{array}{c}\boldsymbol{k}^i\\\overleftarrow{\boldsymbol{l}}\end{array};a\right]y^{\lvert\boldsymbol{k}^i\rvert+a+\lvert\boldsymbol{l}\rvert}*(-1)^i[\boldsymbol{k}_i]_{x,y}^{\star}\\
  &\phantom{{}={}}+(-1)^{r+s+1}\sum_{j=0}^{s}(-1)^j(-1)^{r+j+1}\left[\begin{array}{c}\overleftarrow{\boldsymbol{l}_j}\\\boldsymbol{k}\end{array};a\right]x^{\lvert\boldsymbol{k}\rvert+a+\lvert\boldsymbol{l}_j\rvert}*(-1)^{s-j}[\boldsymbol{l}^j]_{x,y}^{\star}\\
  &=\sum_{i=0}^{r}(-1)^{r-i}\left[\begin{array}{c}\boldsymbol{k}^i\\\overleftarrow{\boldsymbol{l}}\end{array};a\right]y^{\lvert\boldsymbol{k}^i\rvert+a+\lvert\boldsymbol{l}\rvert}*[\boldsymbol{k}_i]_{x,y}^{\star}
    +\sum_{j=0}^{s}(-1)^j\left[\begin{array}{c}\overleftarrow{\boldsymbol{l}_j}\\\boldsymbol{k}\end{array};a\right]x^{\lvert\boldsymbol{k}\rvert+a+\lvert\boldsymbol{l}_j\rvert}*[\boldsymbol{l}^j]_{x,y}^{\star}.
 \end{align*}
 Lemma~\ref{lem:alternating3} shows that
 \begin{align*}
  [\boldsymbol{k},a,\boldsymbol{l}]_{x,y}
  &=\sum_{i=0}^{r}[\boldsymbol{k}_i]*[\overleftarrow{\boldsymbol{l}},a,\overleftarrow{\boldsymbol{k}^i}]x^{\lvert\boldsymbol{k}_i\rvert}y^{\lvert\boldsymbol{k}^i\rvert+a+\lvert\boldsymbol{l}\rvert}
    +\sum_{j=0}^{s}[\boldsymbol{k},a,\boldsymbol{l}_j]*[\overleftarrow{\boldsymbol{l}^j}]x^{\lvert\boldsymbol{k}\rvert+a+\lvert\boldsymbol{l}_j\rvert}y^{\lvert\boldsymbol{l}^j\rvert}\\
  &=\sum_{i=0}^{r}[\boldsymbol{k}_i]*\Biggl(\sum_{i'=i}^{r}(-1)^{r-i'}\left[\begin{array}{c}\overleftarrow{\boldsymbol{l}}\\\boldsymbol{k}^{i'}\end{array};a\right]*[\overleftarrow{\boldsymbol{k}^i_{i'}}]\Biggr)x^{\lvert\boldsymbol{k}_i\rvert}y^{\lvert\boldsymbol{k}^i\rvert+a+\lvert\boldsymbol{l}\rvert}\\
  &\hphantom{{}={}}+\sum_{j=0}^{s}\Biggl(\sum_{j'=0}^{j}(-1)^{j'}\left[\begin{array}{c}\boldsymbol{k}\\\overleftarrow{\boldsymbol{l}_{j'}}\end{array};a\right]*[\boldsymbol{l}^{j'}_j]\Biggr)*[\overleftarrow{\boldsymbol{l}^j}]x^{\lvert\boldsymbol{k}\rvert+a+\lvert\boldsymbol{l}_j\rvert}y^{\lvert\boldsymbol{l}^j\rvert}\\
  &=\sum_{i'=0}^{r}(-1)^{r-i'}\left[\begin{array}{c}\overleftarrow{\boldsymbol{l}}\\\boldsymbol{k}^{i'}\end{array};a\right]y^{\lvert\boldsymbol{k}^{i'}\rvert+a+\lvert\boldsymbol{l}\rvert}*\Biggl(\sum_{i=0}^{i'}[\boldsymbol{k}_i]*[\overleftarrow{\boldsymbol{k}^i_{i'}}]x^{\lvert\boldsymbol{k}_i\rvert}y^{\lvert\boldsymbol{k}^i_{i'}\rvert}\Biggr)\\
  &\phantom{{}={}}+\sum_{j'=0}^{s}(-1)^{j'}\left[\begin{array}{c}\boldsymbol{k}\\\overleftarrow{\boldsymbol{l}_{j'}}\end{array};a\right]x^{\lvert\boldsymbol{k}\rvert+a+\lvert\boldsymbol{l}_{j'}\rvert}*\Biggl(\sum_{j=j'}^{s}[\boldsymbol{l}^{j'}_j]*[\overleftarrow{\boldsymbol{l}^j}]x^{\lvert\boldsymbol{l}^{j'}_j\rvert}y^{\lvert\boldsymbol{l}^j\rvert}\Biggr)\\
  &=\sum_{i'=0}^{r}(-1)^{r-i'}\left[\begin{array}{c}\overleftarrow{\boldsymbol{l}}\\\boldsymbol{k}^{i'}\end{array};a\right]y^{\lvert\boldsymbol{k}^{i'}\rvert+a+\lvert\boldsymbol{l}\rvert}*[\boldsymbol{k}_{i'}]_{x,y}\\
  &\phantom{{}={}}+\sum_{j'=0}^{s}(-1)^{j'}\left[\begin{array}{c}\boldsymbol{k}\\\overleftarrow{\boldsymbol{l}_{j'}}\end{array};a\right]x^{\lvert\boldsymbol{k}\rvert+a+\lvert\boldsymbol{l}_{j'}\rvert}*[\boldsymbol{l}^{j'}]_{x,y},
 \end{align*}
 which completes the proof.
\end{proof}

\subsection{Sum formula for Schur multiple zeta values of anti-hook type}
\begin{thm}[Bachmann, Kadota, Suzuki, Yamamoto, and Yamasaki~\cite{BKSYY}]\label{thm:sumSchur}
 If $r$ and $s$ are nonnegative integers and $w$ is an integer with $w\ge r+s+2$, then
 \[
  \sum_{\substack{k_1+\dots+k_r+a+l_1+\dots+l_s=w\\k_1,\dots,k_r,l_1,\dots,l_s\ge1\\a\ge2}}
  \zeta\left(\begin{array}{c}k_1,\dots,k_r\\l_1,\dots,l_s\end{array};a\right)=\binom{w-1}{s}\zeta(w).
 \]
\end{thm}

\begin{rem}
 This theorem is a common generalization of the identities in Theorem~\ref{thm:sum_formula_MZV}.
\end{rem}

This theorem can be rephrased as follows:
\begin{prop}\label{prop:sumSchurgen}
 We have
 \[
  \sum_{\substack{\boldsymbol{k},\boldsymbol{l}\\a\ge2}}\zeta\left(\begin{array}{c}\boldsymbol{k}\\\boldsymbol{l}\end{array};a\right)A^{\dep\boldsymbol{k}}B^{\dep\boldsymbol{l}}W^{\lvert\boldsymbol{k}\rvert+a+\lvert\boldsymbol{l}\rvert}
  =\frac{W}{1-A}(\psi_1((1+B)W)-\psi_1((A+B)W))
 \]
 in $\mathcal{Z}[A,B][[W]]$.
\end{prop}

\begin{proof}
 Theorem~\ref{thm:sumSchur} shows that
 \begin{align*}
  \sum_{\substack{\boldsymbol{k},\boldsymbol{l}\\a\ge2}}\zeta\left(\begin{array}{c}\boldsymbol{k}\\\boldsymbol{l}\end{array};a\right)A^{\dep\boldsymbol{k}}B^{\dep\boldsymbol{l}}W^{\lvert\boldsymbol{k}\rvert+a+\lvert\boldsymbol{l}\rvert}
  &=\sum_{\substack{r,s\ge0\\w\ge r+s+2}}\sum_{\substack{\lvert\boldsymbol{k}\rvert+a+\lvert\boldsymbol{l}\rvert=w\\\dep\boldsymbol{k}=r,\dep\boldsymbol{l}=s\\a\ge2}}\zeta\left(\begin{array}{c}\boldsymbol{k}\\\boldsymbol{l}\end{array};a\right)A^rB^sW^w\\
  &=\sum_{\substack{r,s\ge0\\w\ge r+s+2}}\binom{w-1}{s}\zeta(w)A^rB^sW^w\\
  &=\sum_{w=2}^{\infty}\sum_{s=0}^{w-2}\sum_{r=0}^{w-s-2}\binom{w-1}{s}A^rB^s\zeta(w)W^w\\
  &=\sum_{w=2}^{\infty}\sum_{s=0}^{w-2}\binom{w-1}{s}\frac{1-A^{w-s-1}}{1-A}B^s\zeta(w)W^w\\
  &=\sum_{w=2}^{\infty}\frac{(1+B)^{w-1}-(A+B)^{w-1}}{1-A}\zeta(w)W^w\\
  &=\frac{W}{1-A}(\psi_1((1+B)W)-\psi_1((A+B)W)),
 \end{align*}
 as required.
\end{proof}

\begin{rem}\label{rem:not_true_[k,l,a]}
 It is \emph{not} the case that
 \[
  \sum_{\substack{\boldsymbol{k},\boldsymbol{l}\\a\ge2}}\left[\begin{array}{c}\boldsymbol{k}\\\boldsymbol{l}\end{array};a\right]A^{\dep\boldsymbol{k}}B^{\dep\boldsymbol{l}}W^{\lvert\boldsymbol{k}\rvert+a+\lvert\boldsymbol{l}\rvert}
  =\frac{W}{1-A}(\psi_{1,\mathcal{I}}((1+B)W)-\psi_{1,\mathcal{I}}((A+B)W))
 \]
 in $\mathcal{I}[A,B][[W]]$.
\end{rem}

\subsection{Relationship between the sum formulas for multiple zeta values and symmetric multiple zeta values}
The following proposition somehow explains the similarity between the sum formulas for multiple zeta(-star) values and symmetric multiple zeta(-star) values:
\begin{prop}\label{prop:relation_between_sum_formulas}
 If $\boldsymbol{k}$ and $\boldsymbol{l}$ are indices and $a$ is an integer with $a\ge2$, then we have
 \begin{align*}
  \zeta_S(\boldsymbol{k},a,\boldsymbol{l})
  &=\sum_{i=0}^{r}(-1)^{r-i}\zeta\left(\begin{array}{c}\overleftarrow{\boldsymbol{l}}\\\boldsymbol{k}^i\end{array};a\right)(-1)^{\lvert\boldsymbol{k}^i\rvert+a+\lvert\boldsymbol{l}\rvert}\zeta_S(\boldsymbol{k}_i)
    +\sum_{j=0}^{s}(-1)^j\zeta\left(\begin{array}{c}\boldsymbol{k}\\\overleftarrow{\boldsymbol{l}_j}\end{array};a\right)\zeta_S(\boldsymbol{l}^j),\\
  \zeta_S^{\star}(\boldsymbol{k},a,\boldsymbol{l})
  &=\sum_{i=0}^{r}(-1)^{r-i}\zeta\left(\begin{array}{c}\boldsymbol{k}^i\\\overleftarrow{\boldsymbol{l}}\end{array};a\right)(-1)^{\lvert\boldsymbol{k}^i\rvert+a+\lvert\boldsymbol{l}\rvert}\zeta_S^{\star}(\boldsymbol{k}_i)
    +\sum_{j=0}^{s}(-1)^j\zeta\left(\begin{array}{c}\overleftarrow{\boldsymbol{l}_j}\\\boldsymbol{k}\end{array};a\right)\zeta_S^{\star}(\boldsymbol{l}^j),
 \end{align*}
 where $r=\dep\boldsymbol{k}$ and $s=\dep\boldsymbol{l}$.
\end{prop}

\begin{proof}
 Apply $Z$ to the identities in Lemma~\ref{lem:key}, and set $x=1$ and $y=-1$.
\end{proof}

We now deduce Theorem~\ref{thm:sum_formula_SMZV} from Theorem~\ref{thm:sumSchur}, which is a generalization of Theorem~\ref{thm:sum_formula_MZV}, with the aid of Proposition~\ref{prop:relation_between_sum_formulas}.
Let $r$ and $s$ be nonnegative integers and let $w$ be an integer with $w\ge r+s+2$.
Then by summing the identities in Proposition~\ref{prop:relation_between_sum_formulas} over all indices $\boldsymbol{k}$ and $\boldsymbol{l}$ and all integers $a\ge2$ satisfying $\dep\boldsymbol{k}=r$, $\dep\boldsymbol{l}=s$, and $\lvert\boldsymbol{k}\rvert+a+\lvert\boldsymbol{l}\rvert=w$ and
by using Theorem~\ref{thm:sumSchur} and the fact that any symmetric sum of $\zeta_S$ of depth greater than $0$ is $0$ modulo $\zeta(2)\mathcal{Z}$, we obtain
\begin{align*}
 &\sum_{\substack{\lvert\boldsymbol{k}\rvert+a+\lvert\boldsymbol{l}\rvert=w\\\dep\boldsymbol{k}=r,\dep\boldsymbol{l}=s\\a\ge2}}\zeta_S(\boldsymbol{k},a,\boldsymbol{l})\\
 &=\sum_{i=0}^{r}(-1)^{r-i}\sum_{\substack{w_1+w_2=w\\w_1\ge2,w_2\ge0}}\Biggl(\sum_{\substack{\lvert\boldsymbol{k}\rvert+a+\lvert\boldsymbol{l}\rvert=w_1\\\dep\boldsymbol{k}=r-i,\dep\boldsymbol{l}=s\\a\ge2}}\zeta\left(\begin{array}{c}\boldsymbol{l}\\\boldsymbol{k}\end{array};a\right)\Biggr)(-1)^{w_1}\Biggl(\sum_{\substack{\lvert\boldsymbol{k}\rvert=w_2\\\dep\boldsymbol{k}=i}}\zeta_S(\boldsymbol{k})\Biggr)\\
 &\phantom{{}={}}+\sum_{j=0}^{s}(-1)^j\sum_{\substack{w_1+w_2=w\\w_1\ge2,w_2\ge0}}\Biggl(\sum_{\substack{\lvert\boldsymbol{k}\rvert+a+\lvert\boldsymbol{l}\rvert=w_1\\\dep\boldsymbol{k}=r,\dep\boldsymbol{l}=j\\a\ge2}}\zeta\left(\begin{array}{c}\boldsymbol{k}\\\boldsymbol{l}\end{array};a\right)\Biggr)\Biggl(\sum_{\substack{\lvert\boldsymbol{l}\rvert=w_2\\\dep\boldsymbol{l}=s-j}}\zeta_S(\boldsymbol{l})\Biggr)\\
 &\equiv(-1)^r\Biggl(\sum_{\substack{\lvert\boldsymbol{k}\rvert+a+\lvert\boldsymbol{l}\rvert=w\\\dep\boldsymbol{k}=r,\dep\boldsymbol{l}=s\\a\ge2}}\zeta\left(\begin{array}{c}\boldsymbol{l}\\\boldsymbol{k}\end{array};a\right)\Biggr)(-1)^w
  +(-1)^s\Biggl(\sum_{\substack{\lvert\boldsymbol{k}\rvert+a+\lvert\boldsymbol{l}\rvert=w\\\dep\boldsymbol{k}=r,\dep\boldsymbol{l}=s\\a\ge2}}\zeta\left(\begin{array}{c}\boldsymbol{k}\\\boldsymbol{l}\end{array};a\right)\Biggr)\\
 &=(-1)^r\binom{w-1}{r}\zeta(w)(-1)^w+(-1)^s\binom{w-1}{s}\zeta(w)\\
 &\equiv\biggl(-(-1)^r\binom{w-1}{r}+(-1)^s\binom{w-1}{s}\biggr)\zeta(w)
\end{align*}
and
\begin{align*}
 &\sum_{\substack{\lvert\boldsymbol{k}\rvert+a+\lvert\boldsymbol{l}\rvert=w\\\dep\boldsymbol{k}=r,\dep\boldsymbol{l}=s\\a\ge2}}\zeta_S^{\star}(\boldsymbol{k},a,\boldsymbol{l})\\
 &=\sum_{i=0}^{r}(-1)^{r-i}\sum_{\substack{w_1+w_2=w\\w_1\ge2,w_2\ge0}}\Biggl(\sum_{\substack{\lvert\boldsymbol{k}\rvert+a+\lvert\boldsymbol{l}\rvert=w_1\\\dep\boldsymbol{k}=r-i,\dep\boldsymbol{l}=s\\a\ge2}}\zeta\left(\begin{array}{c}\boldsymbol{k}\\\boldsymbol{l}\end{array};a\right)\Biggr)(-1)^{w_1}\Biggl(\sum_{\substack{\lvert\boldsymbol{k}\rvert=w_2\\\dep\boldsymbol{k}=i}}\zeta_S^{\star}(\boldsymbol{k})\Biggr)\\
 &\phantom{{}={}}+\sum_{j=0}^{s}(-1)^j\sum_{\substack{w_1+w_2=w\\w_1\ge2,w_2\ge0}}\Biggl(\sum_{\substack{\lvert\boldsymbol{k}\rvert+a+\lvert\boldsymbol{l}\rvert=w_1\\\dep\boldsymbol{k}=r,\dep\boldsymbol{l}=j\\a\ge2}}\zeta\left(\begin{array}{c}\boldsymbol{l}\\\boldsymbol{k}\end{array};a\right)\Biggr)\Biggl(\sum_{\substack{\lvert\boldsymbol{l}\rvert=w_2\\\dep\boldsymbol{l}=s-j}}\zeta_S^{\star}(\boldsymbol{l})\Biggr)\\
 &\equiv(-1)^r\Biggl(\sum_{\substack{\lvert\boldsymbol{k}\rvert+a+\lvert\boldsymbol{l}\rvert=w\\\dep\boldsymbol{k}=r,\dep\boldsymbol{l}=s\\a\ge2}}\zeta\left(\begin{array}{c}\boldsymbol{k}\\\boldsymbol{l}\end{array};a\right)\Biggr)(-1)^w
  +(-1)^s\Biggl(\sum_{\substack{\lvert\boldsymbol{k}\rvert+a+\lvert\boldsymbol{l}\rvert=w\\\dep\boldsymbol{k}=r,\dep\boldsymbol{l}=s\\a\ge2}}\zeta\left(\begin{array}{c}\boldsymbol{l}\\\boldsymbol{k}\end{array};a\right)\Biggr)\\
 &=(-1)^r\binom{w-1}{s}\zeta(w)(-1)^w+(-1)^s\binom{w-1}{r}\zeta(w)\\
 &\equiv\biggl((-1)^s\binom{w-1}{r}-(-1)^r\binom{w-1}{s}\biggr)\zeta(w)
\end{align*}
modulo $\zeta(2)\mathcal{Z}$.

\subsection{Proof of our main theorem}
We begin with computing the generating functions
\[
 \sum_{\substack{\boldsymbol{k},\boldsymbol{l}\\a\ge2}}[\boldsymbol{k},a,\boldsymbol{l}]_{x,y}A^{\dep\boldsymbol{k}}B^{\dep\boldsymbol{l}}W^{\lvert\boldsymbol{k}\rvert+a+\lvert\boldsymbol{l}\rvert},
 \sum_{\substack{\boldsymbol{k},\boldsymbol{l}\\a\ge2}}[\boldsymbol{k},a,\boldsymbol{l}]_{x,y}^{\star}A^{\dep\boldsymbol{k}}B^{\dep\boldsymbol{l}}W^{\lvert\boldsymbol{k}\rvert+a+\lvert\boldsymbol{l}\rvert}\in\mathcal{I}[x,y][A,B][[W]].
\]
Since it is unlikely that the generating functions can be written in terms of $\Gamma_{1.\mathcal{I}}(W)$ only, we shall use the generating function
\[
 F_{\mathcal{I}}(A,B,W)
 =\sum_{\substack{\boldsymbol{k},\boldsymbol{l}\\a\ge2}}\left[\begin{array}{c}\boldsymbol{k}\\\boldsymbol{l}\end{array};a\right]A^{\dep\boldsymbol{k}}B^{\dep\boldsymbol{l}}W^{\lvert\boldsymbol{k}\rvert+a+\lvert\boldsymbol{l}\rvert}
 \in\mathcal{I}[A,B][[W]],
\]
which appeared in Remark~\ref{rem:not_true_[k,l,a]}.
Note that
\begin{align*}
 \tilde{S}(F_{\mathcal{I}}(A,B,W))
 &=\sum_{\substack{\boldsymbol{k},\boldsymbol{l}\\a\ge2}}\tilde{S}\left(\left[\begin{array}{c}\boldsymbol{k}\\\boldsymbol{l}\end{array};a\right]\right)A^{\dep\boldsymbol{k}}B^{\dep\boldsymbol{l}}W^{\lvert\boldsymbol{k}\rvert+a+\lvert\boldsymbol{l}\rvert}\\
 &=\sum_{\substack{\boldsymbol{k},\boldsymbol{l}\\a\ge2}}(-1)^{\dep\boldsymbol{k}+\dep\boldsymbol{l}+1}\left[\begin{array}{c}\boldsymbol{l}\\\boldsymbol{k}\end{array};a\right]A^{\dep\boldsymbol{k}}B^{\dep\boldsymbol{l}}W^{\lvert\boldsymbol{k}\rvert+a+\lvert\boldsymbol{l}\rvert}\\
 &=-F_{\mathcal{I}}(-B,-A,W).
\end{align*}

\begin{prop}\label{prop:gen_func_[k,a,l]_xy}
 We have
 \begin{align*}
  \sum_{\substack{\boldsymbol{k},\boldsymbol{l}\\a\ge2}}[\boldsymbol{k},a,\boldsymbol{l}]_{x,y}A^{\dep\boldsymbol{k}}B^{\dep\boldsymbol{l}}W^{\lvert\boldsymbol{k}\rvert+a+\lvert\boldsymbol{l}\rvert}
  &=F_{\mathcal{I}}(B,-A,yW)\frac{\Gamma_{1,\mathcal{I}}(xW)\Gamma_{1,\mathcal{I}}(yW)}{\Gamma_{1,\mathcal{I}}(x(1-A)W)\Gamma_{1,\mathcal{I}}(y(1-A)W)}\\
  &\hphantom{{}={}}+F_{\mathcal{I}}(A,-B,xW)\frac{\Gamma_{1,\mathcal{I}}(xW)\Gamma_{1,\mathcal{I}}(yW)}{\Gamma_{1,\mathcal{I}}(x(1-B)W)\Gamma_{1,\mathcal{I}}(y(1-B)W)},\\
  \sum_{\substack{\boldsymbol{k},\boldsymbol{l}\\a\ge2}}[\boldsymbol{k},a,\boldsymbol{l}]_{x,y}^{\star}A^{\dep\boldsymbol{k}}B^{\dep\boldsymbol{l}}W^{\lvert\boldsymbol{k}\rvert+a+\lvert\boldsymbol{l}\rvert}
  &=F_{\mathcal{I}}(-A,B,yW)\frac{\Gamma_{1,\mathcal{I}}(x(1+A)W)\Gamma_{1,\mathcal{I}}(y(1+A)W)}{\Gamma_{1,\mathcal{I}}(xW)\Gamma_{1,\mathcal{I}}(yW)}\\
  &\hphantom{{}={}}+F_{\mathcal{I}}(-B,A,xW)\frac{\Gamma_{1,\mathcal{I}}(x(1+B)W)\Gamma_{1,\mathcal{I}}(y(1+B)W)}{\Gamma_{1,\mathcal{I}}(xW)\Gamma_{1,\mathcal{I}}(yW)}\\
 \end{align*}
 in $\mathcal{I}[x,y][A,B][[W]]$.
\end{prop}

\begin{proof}
 The first identity implies the second because
 \begin{align*}
  &\sum_{\substack{\boldsymbol{k},\boldsymbol{l}\\a\ge2}}[\boldsymbol{k},a,\boldsymbol{l}]_{x,y}^{\star}A^{\dep\boldsymbol{k}}B^{\dep\boldsymbol{l}}W^{\lvert\boldsymbol{k}\rvert+a+\lvert\boldsymbol{l}\rvert}\\
  &=-\tilde{S}\Biggl(\sum_{\substack{\boldsymbol{k},\boldsymbol{l}\\a\ge2}}[\boldsymbol{k},a,\boldsymbol{l}]_{x,y}(-A)^{\dep\boldsymbol{k}}(-B)^{\dep\boldsymbol{l}}W^{\lvert\boldsymbol{k}\rvert+a+\lvert\boldsymbol{l}\rvert}\Biggr)\\
  &=-\tilde{S}\biggl(F_{\mathcal{I}}(-B,A,yW)\frac{\Gamma_{1,\mathcal{I}}(xW)\Gamma_{1,\mathcal{I}}(yW)}{\Gamma_{1,\mathcal{I}}(x(1+A)W)\Gamma_{1,\mathcal{I}}(y(1+A)W)}\\
  &\hphantom{{}=-\tilde{S}\biggl(}+F_{\mathcal{I}}(-A,B,xW)\frac{\Gamma_{1,\mathcal{I}}(xW)\Gamma_{1,\mathcal{I}}(yW)}{\Gamma_{1,\mathcal{I}}(x(1+B)W)\Gamma_{1,\mathcal{I}}(y(1+B)W)}\biggr)\\
  &=F_{\mathcal{I}}(-A,B,yW)\frac{\Gamma_{1,\mathcal{I}}(x(1+A)W)\Gamma_{1,\mathcal{I}}(y(1+A)W)}{\Gamma_{1,\mathcal{I}}(xW)\Gamma_{1,\mathcal{I}}(yW)}+F_{\mathcal{I}}(-B,A,xW)\frac{\Gamma_{1,\mathcal{I}}(x(1+B)W)\Gamma_{1,\mathcal{I}}(y(1+B)W)}{\Gamma_{1,\mathcal{I}}(xW)\Gamma_{1,\mathcal{I}}(yW)}.
 \end{align*}
 Lemma~\ref{lem:key} shows that
 \begin{align*}
  &\sum_{\substack{\boldsymbol{k},\boldsymbol{l}\\a\ge2}}[\boldsymbol{k},a,\boldsymbol{l}]_{x,y}A^{\dep\boldsymbol{k}}B^{\dep\boldsymbol{l}}W^{\lvert\boldsymbol{k}\rvert+a+\lvert\boldsymbol{l}\rvert}\\
  &=\sum_{r,s=0}^{\infty}\sum_{\substack{\dep\boldsymbol{k}=r,\dep\boldsymbol{l}=s\\a\ge2}}[\boldsymbol{k},a,\boldsymbol{l}]_{x,y}A^rB^sW^{\lvert\boldsymbol{k}\rvert+a+\lvert\boldsymbol{l}\rvert}\\
  &=\sum_{r,s=0}^{\infty}\sum_{\substack{\dep\boldsymbol{k}=r,\dep\boldsymbol{l}=s\\a\ge2}}\sum_{i=0}^{r}(-1)^{r-i}\left[\begin{array}{c}\overleftarrow{\boldsymbol{l}}\\\boldsymbol{k}^i\end{array};a\right]y^{\lvert\boldsymbol{k}^i\rvert+a+\lvert\boldsymbol{l}\rvert}*[\boldsymbol{k}_i]_{x,y}A^rB^sW^{\lvert\boldsymbol{k}\rvert+a+\lvert\boldsymbol{l}\rvert}\\
  &\hphantom{{}={}}+\sum_{r,s=0}^{\infty}\sum_{\substack{\dep\boldsymbol{k}=r,\dep\boldsymbol{l}=s\\a\ge2}}\sum_{j=0}^{s}(-1)^j\left[\begin{array}{c}\boldsymbol{k}\\\overleftarrow{\boldsymbol{l}_j}\end{array};a\right]x^{\lvert\boldsymbol{k}\rvert+a+\lvert\boldsymbol{l}_j\rvert}*[\boldsymbol{l}^j]_{x,y}A^rB^sW^{\lvert\boldsymbol{k}\rvert+a+\lvert\boldsymbol{l}\rvert}\\
  &=\Biggl(\sum_{i,s=0}^{\infty}\sum_{\substack{\dep\boldsymbol{k}=i,\dep\boldsymbol{l}=s\\a\ge2}}\left[\begin{array}{c}\boldsymbol{l}\\\boldsymbol{k}\end{array};a\right](-A)^iB^s(yW)^{\lvert\boldsymbol{k}\rvert+a+\lvert\boldsymbol{l}\rvert}\Biggr)\Biggl(\sum_{i=0}^{\infty}\sum_{\dep\boldsymbol{k}=i}[\boldsymbol{k}]_{x,y}A^iW^{\lvert\boldsymbol{k}\rvert}\Biggr)\\
  &\hphantom{{}={}}+\Biggl(\sum_{r,j=0}^{\infty}\sum_{\substack{\dep\boldsymbol{k}=r,\dep\boldsymbol{l}=j\\a\ge2}}\left[\begin{array}{c}\boldsymbol{k}\\\boldsymbol{l}\end{array};a\right]A^r(-B)^j(xW)^{\lvert\boldsymbol{k}\rvert+a+\lvert\boldsymbol{l}\rvert}\Biggr)\Biggl(\sum_{j=0}^{\infty}\sum_{\dep\boldsymbol{l}=j}[\boldsymbol{l}]_{x,y}B^jW^{\lvert\boldsymbol{l}\rvert}\Biggr)\\
  &=\Biggl(\sum_{\substack{\boldsymbol{k},\boldsymbol{l}\\a\ge2}}\left[\begin{array}{c}\boldsymbol{l}\\\boldsymbol{k}\end{array};a\right](-A)^{\dep\boldsymbol{k}}B^{\dep\boldsymbol{l}}(yW)^{\lvert\boldsymbol{k}\rvert+a+\lvert\boldsymbol{l}\rvert}\Biggr)\Biggl(\sum_{\boldsymbol{k}}[\boldsymbol{k}]_{x,y}A^{\dep\boldsymbol{k}}W^{\lvert\boldsymbol{k}\rvert}\Biggr)\\
  &\hphantom{{}={}}+\Biggl(\sum_{\substack{\boldsymbol{k},\boldsymbol{l}\\a\ge2}}\left[\begin{array}{c}\boldsymbol{k}\\\boldsymbol{l}\end{array};a\right]A^{\dep\boldsymbol{k}}(-B)^{\dep\boldsymbol{l}}(xW)^{\lvert\boldsymbol{k}\rvert+a+\lvert\boldsymbol{l}\rvert}\Biggr)\Biggl(\sum_{\boldsymbol{l}}[\boldsymbol{l}]_{x,y}B^{\dep\boldsymbol{l}}W^{\lvert\boldsymbol{l}\rvert}\Biggr)\\
  &=F_{\mathcal{I}}(B,-A,yW)\frac{\Gamma_{1,\mathcal{I}}(xW)\Gamma_{1,\mathcal{I}}(yW)}{\Gamma_{1,\mathcal{I}}(x(1-A)W)\Gamma_{1,\mathcal{I}}(y(1-A)W)}+F_{\mathcal{I}}(A,-B,xW)\frac{\Gamma_{1,\mathcal{I}}(xW)\Gamma_{1,\mathcal{I}}(yW)}{\Gamma_{1,\mathcal{I}}(x(1-B)W)\Gamma_{1,\mathcal{I}}(y(1-B)W)}
 \end{align*}
 by Proposition~\ref{prop:gen_func_kxy}. This completes the proof.
\end{proof}

\begin{thm}[Main theorem]\label{thm:gen_func_zeta_xy(k,a,l)}
 We have
 \begin{align*}
  &\sum_{\substack{\boldsymbol{k},\boldsymbol{l}\\a\ge2}}\zeta_{x,y}(\boldsymbol{k},a,\boldsymbol{l})A^{\dep\boldsymbol{k}}B^{\dep\boldsymbol{l}}W^{\lvert\boldsymbol{k}\rvert+a+\lvert\boldsymbol{l}\rvert}\\
  &=\frac{yW}{1-B}(\psi_1(y(1-A)W)-\psi_1(y(B-A)W))\frac{\Gamma_1(xW)\Gamma_1(yW)}{\Gamma_1(x(1-A)W)\Gamma_1(y(1-A)W)}\\
  &\hphantom{{}={}}+\frac{xW}{1-A}(\psi_1(x(1-B)W)-\psi_1(x(A-B)W))\frac{\Gamma_1(xW)\Gamma_1(yW)}{\Gamma_1(x(1-B)W)\Gamma_1(y(1-B)W)},\\
  &\sum_{\substack{\boldsymbol{k},\boldsymbol{l}\\a\ge2}}\zeta_{x,y}^{\star}(\boldsymbol{k},a,\boldsymbol{l})A^{\dep\boldsymbol{k}}B^{\dep\boldsymbol{l}}W^{\lvert\boldsymbol{k}\rvert+a+\lvert\boldsymbol{l}\rvert}\\
  &=\frac{yW}{1+A}(\psi_1(y(1+B)W)-\psi_1(y(B-A)W))\frac{\Gamma_1(x(1+A)W)\Gamma_1(y(1+A)W)}{\Gamma_1(xW)\Gamma_1(yW)}\\
  &\hphantom{{}={}}+\frac{xW}{1+B}(\psi_1(x(1+A)W)-\psi_1(x(A-B)W))\frac{\Gamma_1(x(1+B)W)\Gamma_1(y(1+B)W)}{\Gamma_1(xW)\Gamma_1(yW)}
 \end{align*}
 in $\mathcal{Z}[T][x,y][A,B][[W]]$.
\end{thm}

\begin{proof}
 Apply $Z$ to the identities in Proposition~\ref{prop:gen_func_[k,a,l]_xy}, and having
 \[
  Z(F_{\mathcal{I}}(A,B,W))=\sum_{\substack{\boldsymbol{k},\boldsymbol{l}\\a\ge2}}\zeta\left(\begin{array}{c}\boldsymbol{k}\\\boldsymbol{l}\end{array};a\right)A^{\dep\boldsymbol{k}}B^{\dep\boldsymbol{l}}W^{\lvert\boldsymbol{k}\rvert+a+\lvert\boldsymbol{l}\rvert}
 \]
 in mind, use Proposition~\ref{prop:sumSchurgen}.
\end{proof}

\section*{Acknowledgements}
This work was supported by JSPS KAKENHI Grant Numbers JP18J00982, JP18K03243, and JP18K13392.

\end{document}